\documentclass[pdflatex,sn-mathphys-num]{sn-jnl}%

\usepackage{geometry}
\usepackage{graphicx}%
\usepackage{multirow}%
\usepackage{amsmath,amssymb,amsfonts}%
\usepackage{amsthm}%
\usepackage{mathrsfs}%

\usepackage{enumitem}
\usepackage{url,titletoc,mathtools}
\usepackage{tikz-cd}
\usepackage{hyperref}
\hypersetup{
linkbordercolor = {cyan}}

\DeclareMathOperator{\tr}{tr}

\DeclareMathOperator{\Gal}{Gal}

\DeclareMathOperator{\Aut}{Aut}

\DeclareMathOperator{\Spec}{Spec}

\DeclareMathOperator{\Jac}{Jac}

\newcommand{\frob}[1]{\sigma_#1}
\newcommand{\U}[1]{U_#1}
\newcommand{\Z}[1]{Z_#1}

\def\A{{\cO_{K}}}
\def\B{{\cO_{L}}}
\def\F{K}
\def\L{{L}}
\def\RF{{\mathbb F_K}}
\def\RL{{\mathbb F_{L}}}
\def\SS{S}

\def\rhol{\delta}

\newcommand{\rhoid}[1]{\rho_{f,H_#1, H_{1,#1}}} 
\newcommand{\rhotype}[1]{\rho_{f,#1}}
\newcommand{\Uid}[1]{U_{f,H_#1, H_{1,#1}}}
\newcommand{\Utype}[1]{U_{f,#1}}

\def\bbQ{ {\mathbb Q}}

\def\bb1{ {\mathbb 1}}

\def\cI{ {\mathcal I} }
\def\cJ{ {\mathcal J}}

\def\cL{ {\mathcal L} }

\def\cM{{\mathcal M}}

\def\cO{ {\mathcal O} }
\def\cP{{\mathcal P}}

\def\cS{{\mathcal S}}

\newcounter{lst}

\theoremstyle{plain}

\newtheorem{theorem}{theorem}[section]

\newtheorem{lemma}[theorem]{Lemma}

\newtheorem{conjecture}[theorem]{Conjecture}
\newtheorem{thm}[theorem]{Theorem}

\newtheorem{corollary}[theorem]{Corollary}
\newtheorem*{lemma*}{Lemma}

\theoremstyle{definition}

\newtheorem{definition}[theorem]{Definition}

\newtheorem{hypothesis}[theorem]{Hypothesis}

\newtheorem{example}[theorem]{Example}
\newtheorem{remark}[theorem]{Remark}

\numberwithin{equation}{section}

\usepackage{xcolor}
\newcommand{\revstart}{\begingroup}%
\newcommand{\revend}{\endgroup}
\begin{document}

\title{A Chebotarev Density Theorem over p-adic Fields}

\author[1]{\fnm{Asvin}\sur{G}\email{gasvinseeker94@gmail.com}}
\author[1]{\fnm{Yifan}\sur{Wei}\email{yifan.wei@wisc.edu}}
\author[1]{\fnm{John}\sur{Yin}\email{yin.1034@osu.edu}}
\affil[1]{\orgdiv{Department of Mathematics}, \orgname{University of Wisconsin-Madison}, \orgaddress{\street{480 Lincoln Dr}, \city{Madison}, \postcode{53706}, \state{WI}, \country{U.S}}}

\abstract{
\revstart
    We compute the $p$-adic densities of points with a given splitting type along a (generically) finite map, analogous to the classical Chebotarev theorem over number fields and function fields. Under a resolution hypothesis, we prove that the paired densities $\rhotype{\tau}+\rhotype{\tau^{-1}}$ satisfy a functional equation in the size of the residue field. This functional equation is a direct reflection of Poincar\'e duality in \'etale cohomology. As a consequence, we prove the tame case of the $\rho$-part of a conjecture of Bhargava, Cremona, Fisher, and Gajovi\'c on factorization densities of p-adic polynomials: the case where the residue characteristic is coprime to the ramification indices of the factorization type. 
\revend
    
\revstart
    The key tool is the notion of \emph{admissible pairs} associated to a group, which we use as an invariant of the inertia and decomposition action of a local field on the fibers of the finite map. We compute the splitting densities by M\"obius inverting certain p-adic integrals along the poset of admissible pairs. The $\rho$-part of the conjecture on factorization densities follows for tamely ramified types from our general results combined with known rationality results; the wild case is proved by a different method in a companion paper. 
\revend
}

\maketitle

\tableofcontents

\section{Introduction}

\revstart
Let $h(z) = c_nz^n + c_{n-1}z^{n-1} + \dots + c_0$ be a random polynomial having coefficients $ c_i \in \mathbb Z_p$. In this paper, we develop a general method to compute the density of polynomials with a given factorization type over $\mathbb Z_p$, thus proving the tame case of the $\rho$-part of a conjecture of Bhargava, Cremona, Fisher and Gajovi\'c \cite{bhargava2022density}. As we will see, our method is far more general than this special case and can be considered as a Chebotarev-type theorem for finite maps over local fields (as opposed to the classical versions which apply to varieties over finitely generated rings). For concreteness, we first explain the case of a random polynomial.
\revend

We parametrize polynomials $h(z)$ as above by $\mathbb Z_p^{n+1}$ with the Haar measure $\mu_{\mathrm{Haar}}$ normalized so that the total measure is $1$. If $h(z)$ is irreducible, define $K = \mathbb Q_p[z]/(h(z))$ and the \emph{factorization type} of $h$ by $\sigma(h) = \{f^e\}$ where $f$ is the residue degree of $K$ and $e$ its ramification index. If $e=1$, we will often omit the superscript and simply write $\{f\}$ instead of $\{f^1\}$. In general, if $h(z)$ is squarefree with a factorization $h = g_1\dots g_r$ into irreducible polynomials over $\mathbb Q_p$, we define its factorization type to be the multiset $\sigma(h) = \{\sigma(g_1),\dots,\sigma(g_r)\}$. 

Fixing such a factorization type $\sigma$, let $U_\sigma(p) \subset \mathbb Z_p^{n+1}$ be the $p$-adic open subset of squarefree polynomials with factorization type $\sigma$ and $\rho(n,\sigma;p) = \mu_{\mathrm{Haar}}(U_\sigma(p))$. As a few sample examples, \cite{bhargava2022density} computes
\vspace{5 mm}
\begin{center}
\begin{tabular}{ |c||c| } 
\hline
 $\rho(2,(11);p) = \frac{1}{2}$  &$\rho(2,(2);p) = \frac{p^2-p+1}{2(p^2+p+1)}$\\
 $\rho(3,(111);p) = \frac{p^4+2p^2+1}{6(p^4+p^3+p^2+p+1)}$ & $\rho(3,(12);p) = \frac{p^4+1}{2(p^4+p^3+p^2+p+1)}$\\
 \hline
\end{tabular}
\end{center}
\vspace{5 mm}
Based on numerical evidence and some of their results, \cite[Conjecture 1.2]{bhargava2022density} is the following
\begin{conjecture}\label{conj: intro bhargava}
The densities $\rho(n,\sigma;p)$ are rational functions in $p$ and satisfy the following symmetry:
\[\rho(n,\sigma;p^{-1}) = \rho(n,\sigma;p).\]
\end{conjecture}

In more detail, consider the random variable that sends a polynomial to the number of $\mathbb Q_p$-rational roots. Then, \cite[Theorem 1]{bhargava2022density} shows that all the moments of this random variable are rational functions in $p$ that are invariant under the transformation $p \to p^{-1}$ and they conjecture that in fact, the densities themselves satisfy both properties. Since the number of rational roots is far from determining the factorization type of a polynomial, the conjecture is much stronger than what they prove.

Also in prior work, \cite{caruso2022zeroes} considers $\rho_n(x)$ as a density measuring the expected number of roots of a degree $n$ polynomial near a point $x$ in a p-adic field. He proves that these densities satisfy several striking properties such as a modular transformation law. 

In the case where the splitting fields of the polynomials are tame extensions of $\mathbb{Q}_p$, \cite{del_corso_dvornicich_2000} (and independently, the last author in \cite{john}) prove that the densities $\rho(n,\sigma;p)$ form a rational function in $p$. They compute a recurrence equation for the above densities in terms of certain other densities. The individual terms in this recurrence are rational but do not satisfy the above symmetry and therefore, we need a new idea in order to prove the functional equation. More recently (and in fact simultaneously with the current paper), \cite{roy} proves the above conjecture including the symmetry for the extremal factorization type $\{n^1\}$.

\revstart
In this paper, we prove the tame case of the $\rho$-part of the conjecture, formulated for any finite extension of $\bbQ_p$ whose residue characteristic is coprime to the ramification indices occurring in $\sigma$ (Theorem \ref{thm: symmetry follows from rationality for factorization densities}); it follows from our main general results together with the previous work on rationality cited above. The functional equation ultimately reduces to the Weil Conjectures (more specifically, Poincar\'e duality) over finite fields. The key insight that led to this paper was decoupling the two parts of the conjecture: the rationality turns out to be very special to polynomial factorizations, but the functional equation, suitably generalized, applies to the paired densities $\rhotype{\tau}+\rhotype{\tau^{-1}}$ for generically finite Galois maps between smooth, proper varieties over the ring of integers of a finite extension of $\mathbb Q_p$, under the resolution hypothesis, and offers us an inductive lever to prove all our results.
\revend

We thus move from the setting of polynomials to algebraic geometry and leverage p-adic integration to prove our desired results. While our application of p-adic integration is far from novel in this setting, our inductive strategy (\S\ref{section: admissible pairs}) does appear to be novel, as well as Lemma \ref{lemma: local integrals}, which allows us to express a very general base change formula for p-adic integration along finite covers.

\revstart
Before stating our general results, we briefly recall the classical Chebotarev theorem since our results can be considered as a refinement of that theorem for local fields. Over a finite field $\mathbb{F}_q$, let $f: X \to Y$ be an \'etale Galois map between geometrically irreducible varieties of positive dimension with constant Galois group $G$ acting simply and transitively on geometric fibers. 
\revend

Given a closed point $y \in Y$ with residue field $\kappa(y) = \mathbb{F}_{q^d}$ and a geometric lift $x \in X(\overline{\mathbb F}_q)$, the $q^d$-Frobenius will send $x$ to a geometric point $x'$ in the fiber. By simple-transitivity of the $G$-action on the fiber, there exists a unique element $\sigma_y \in G$, which satisfies $\sigma_y(x) = x'$. The conjugacy class of the element $\sigma_y$ will not depend on the choice of $x$ from the fiber, and by the abuse of notation we will call the conjugacy class the \textit{Frobenius element}, which we will also denote by $\sigma_y$. 

The classical Chebotarev states that the association $y \leadsto \sigma_y$ is uniformly random on $G$:

\begin{thm}[Chebotarev]
Let $c \subset G$ be a conjugacy class. Then,
\[\lim_{N \to \infty}\frac{|\{y \in Y : [\kappa(y):\mathbb F_q] \leq N, \frob{y} = c\}|}{|\{y \in Y : [\kappa(y):\mathbb F_q] \leq N\}|} = \frac{|c|}{|G|}.\]
\end{thm}

\subsection{Our Main Results}

\revstart
We work over the ring of integers $\A$ of a finite extension $\F$ of $\mathbb Q_p$ and residue field $\A/\mathfrak m_{\F} = \mathbb F_q$. Let $f: X \to Y$ be a generically finite, Galois map between smooth proper varieties defined over $\A$ with Galois group $G$, a constant group scheme over $\A$. By this, we mean that there is a Zariski open subset $\U{f} \subset Y$ such that $f: f^{-1}(\U{f}) \to \U{f}$ is an \'etale Galois map with Galois group $G$ and moreover, the natural map $\Aut(f) \to G$ is an isomorphism. 
\revend

Let $P \in \U{f}(\F)$ with $Q \in \U{f}(\overline{\F})$ a geometric lift of $P$. By functoriality of the \'etale fundamental group, we have a map
\[\Gal(\overline{\F}/\F) = \pi_1^{\text{\'et}}(P;Q) \to G.\]
We denote the image of the entire Galois group by $D_Q$, the image of the inertia group by $I_Q$ and the image of the canonical Frobenius coset by $\frob{Q}$. The Galois group of a local field is significantly more complicated than that of a finite field and thus, we need a new definition to capture the data of $(I_Q,\frob{Q})$. We call a pair $\tau = (H,gH)$ for a subgroup $H \subset G$ and an element $ g \in G$ an \emph{admissible pair} if $gH = Hg$ (in this general setting, $\tau$ will refer to more general admissible pairs as opposed to factorization types). The pair $(I_Q,\frob{Q})$ as above is an admissible pair and we define 
\[U_{f,\tau}(\F) = \{P \in \U{f}(\F) : (I_Q,\frob{Q}) = \tau \text{ for some } Q \text{ with } f(Q) = P \}.\]
\revstart
For a smooth projective variety $Y$ as above, one can define a canonical measure $\mu_Y$, see \S\ref{section: canonical measure} for more details. We normalize the canonical measure $\mu_Y$ on $Y(\F)$ so that $\mu_Y(Y(\F)) = |Y(\mathbb F_q)|$. In analogy with the factorization densities conjecture, one might hope that $\rhotype{\tau}(\F) \coloneqq \mu_Y(U_{f,\tau}(\F))$ is a rational function in $q$ as we vary $\A$ over extensions of $\mathbb Z_p$. The simplest examples (such as the degree $2$ map from an elliptic curve $E \to \mathbb P^1$) very quickly disabuse us of this hope or any reasonable modification of it. Under the resolution hypothesis below, we show that $\rhotype{\tau}(\F)$ is a linear combination of the point counts $|Z(\mathbb F_q)|$ weighted by products of local factors $\eta_k(q^s)$ as $Z$ ranges over smooth proper varieties of dimension at most $\dim Y$ where $\eta_k(q) = (q^{k+1}-q)/(1-q^{k+1})$. The obstruction to rationality is precisely that the point counts $|Z(\mathbb F_q)|$ are generally not rational functions of $q$.
\revend

\revstart
Despite this, the symmetry in Conjecture \ref{conj: intro bhargava} mentioned above continues to hold in this general context when formulated correctly. More precisely, we define the notion of a palindromic form of weight $k$ over $\A$ in Definition \ref{defn: palindromic forms} and prove that the paired densities $\rhotype{\tau}+\rhotype{\tau^{-1}}$ are palindromic forms under the resolution hypothesis. For the introduction, we will only need that palindromic forms of weight $k$ are functions $\rho: \mathbb N \to \mathbb Q$ of a special form that allow a natural extension to a function $\rho: \mathbb Z \to \mathbb Q$ with the property that
\revend
\[\rho(-m) = q^{-km}\rho(m).\]

With this definition, our main theorem is as follows. Note that it is not even obvious a-priori that the densities $\rhotype{\tau}(\L)$ lie in $\mathbb Q$.

\begin{thm}[Theorem \ref{thm: Main theorem, density fixing inertia and frobenius coset}]
\revstart
Let $f: X \to Y/\Spec\A$ be a generically finite, Galois map with Galois group $G$ between smooth, projective varieties and $\tau = (H,gH)$ an admissible pair as above. We denote its inverse by $\tau^{-1} = (H,g^{-1}H)$. Suppose that, for every $\tau' = (H',g'H')$ with $H' \subset H$ and $g' \in gH$, there exists a $g'$-equivariant resolution $\pi_{H'}: \widetilde{X_{H'}} \to X/H'$ that is an isomorphism over $(X/H')\setminus f_{H'}(Z_f)$ and for which $\pi_{H'}^{-1}(f_{H'}(Z_f))$ is a simple normal crossings divisor over an unramified extension of $\A$, where $f_{H'}:X\to X/H'$ is the quotient map. 
\revend

Then, for any finite extension of local rings $\A \subset \B$ with corresponding residue field extensions $\mathbb F_q \subset \mathbb F_{q^m}$, the densities $\rhotype{\tau}(\L) \in \mathbb Q$ and the function
\[ m \to \rhotype{\tau}(\L) + \rhotype{\tau^{-1}}(\L)\]
depends on $\L$ only through the size of the residue field $\mathbb F_{q^m}$ and is a palindromic form (in the variable $m$) of weight $\dim_{\F} Y_{\F}$.
\end{thm}

\revstart
One can consider this as a refinement of the classical Chebotarev density theorem over finite fields as $q \to \infty$ (Corollary \ref{cor: q to infty limiting behaviour}). In this limit, after dividing by $q^{\dim Y}$, the ramified densities tend to $0$ and the unramified densities tend to the classical limits. Moreover, combined with the known rationality results, this general theorem implies the symmetry in Conjecture \ref{conj: intro bhargava} for all but finitely many primes, as we explain later.
\revend

Simple examples show that such a resolution need not exist (see the discussion following Hypothesis \ref{hypothesis:res of sing}). If $f: X \to Y$ is generically finite with Galois group $G$ over (an open subset) of the ring of integers of a number field, then the resolution hypothesis is satisfied for all but finitely many completions of the number ring since we can equivariantly resolve singularities over characteristic $0$ (for instance, see \cite{abramovich1996equivariant},\cite{encinas1998good} and \cite{bierstone1997canonical} among many others) and \emph{spread out} to all but finitely many primes as we show in Theorem \ref{thm: main theorem for almost all primes}.

Moreover, the densities of the $\Utype{\tau}(\L)$ individually are not palindromic forms of any weight. The key here is Lemma \ref{lemma: twisted point counts have the right weight} while Remark \ref{rmk: individual densities are not palindromic} shows that $\rhotype{\tau}$ is not palindromic by itself in general. Similarly, the densities where we require that the splitting field is a fixed extension are not palindromic in general. As Caruso explains immediately following \cite[Theorem C]{caruso2022zeroes}, the expected number of roots generating a fixed totally ramified quadratic extension $K$ of a random $p$-adic degree $n$ polynomial (for $n \geq 3$) is $\|D_K\|\frac{p^2(p^2+1)}{p^4+p^3+p^2+p+1}$, which is not invariant under the $p \to p^{-1}$ transformation unless $\|D_K\| = p^{-1}$, i.e., unless $K$ is tamely ramified, but this symmetry is regained when we consider instead the expected number of roots generating any ramified quadratic extension. 

In the context of a random polynomial over $\mathbb Z_p$, the relevant map is
\[f: X = (\mathbb P^1)^n \to (\mathbb P^1)^n/S_n = \mathbb P^n = Y; ([t_1,s_1],\dots,[t_n,s_n]) \to [a_0, \dots, a_n]\]
\revstart
where $(s_1z-t_1)\dots(s_nz-t_n) = \sum a_i z^i$. For any finite extension $K$ of $\mathbb Q_p$ with residue field $\mathbb F_q$, one can define an analogous notion of factorization density $\rho(n,\sigma;q)$. A priori, this might depend on $K$. For tame factorization types, those whose ramification indices are coprime to the residue characteristic, the density is a rational function of $q$ alone (Lemma \ref{lemma: uniformity of factorization densities}, which combines the rationality results cited above with the general theorems of this paper), and we prove that this rational function satisfies the following functional equation.
\revend

\begin{thm}[Theorem \ref{thm: symmetry follows from rationality for factorization densities}]
\revstart
Let $\sigma$ be a factorization type of degree $n$ and let $\mathcal L_\sigma$ be the set of finite extensions of $\mathbb Q_p$, as $p$ varies, whose residue characteristic is coprime to the ramification indices occurring in $\sigma$. There is a rational function $r_\sigma(t) \in \mathbb Q(t)$ with $\rho(n,\sigma;q) = r_\sigma(q)$ for every $K \in \mathcal L_\sigma$ with residue field of size $q$, and it satisfies the symmetry
\revend
\[r_\sigma(t^{-1}) = r_\sigma(t).\]
\end{thm}

\begin{remark}[A motivic analogue]\label{rmk: motivic duality}

The main formulas are obviously of a "motivic" nature and it is tempting to find a motivic analogue of our results. There are several difficulties in this translation that we are tackling in ongoing work.

\end{remark}

\begin{remark}[The wild case]\label{rmk: wild case}
    In a previous version of this paper, we claimed a proof of the above theorem for all primes, including the wild ones. This proceeded by constructing a uniform resolution over $\Spec \mathbb Z$ of the ``resultant locus'' parameterizing tuples of polynomials that share a root. 
    
    More precisely, letting $\mathbb P^{n}$ parameterize degree $n$ polynomials as above, and for $n_1,\dots,n_r \in \mathbb Z_{>0}$ a sequence of positive integers, let $R \subset \prod_{i=1}^r\mathbb P^{n_i}$ correspond to the set of polynomials $f_1,\dots,f_r$ of degree $n_1,\dots,n_r$ such that for some $i\neq j$, $f_i$ and $f_j$ share a root over an algebraic closure. 
    
    This is a Zariski closed locus that is highly singular in general, and constructing a resolution of this by a smooth algebraic space would suffice. The earlier version claimed such a construction but during the review process, a fatal gap in the proof at the wild primes was discovered.

\revstart
    In trying to correct this proof using LLMs (specifically Fable and Astra), an alternative proof using completely different ideas was found \cite{companion}, with an exposition and formal Lean proof \href{https://asving.github.io/claude/splitting-type-density/}{here}. This proof applies specifically to the polynomial setting (and not the general Galois case like this paper) but proves rationality and palindromicity without any input from this paper.
\revend
\end{remark}

\subsection{Organization of the paper}

\revstart
In \S \ref{section: preliminaries}, we review some well known theorems that will be useful for us and standardize notations and normalizations. The most important novelty in this section is the notion of a \emph{palindromic form of weight k} (Definition \ref{defn: palindromic forms}). The other novelty is in \S \ref{section: admissible pairs}, where we define the notion of an admissible pair related to a group and a corresponding partial order on it. In \S \ref{section: on the palindromicity}, we prove our main Theorem \ref{thm: Main theorem, density fixing inertia and frobenius coset} and derive some corollaries. This section contains the core arguments of our paper. Finally, in \S \ref{section: factorization types density}, we adapt our general method to prove the tame case of the $\rho$-part of the conjecture of Bhargava, Cremona, Fisher and Gajovi\'c.
\revend

\textit{Acknowledgements.} We would like to thank Jordan Ellenberg, Daniel Litt and Andrew O'Desky for general helpful discussions.

\section{Preliminaries}\label{section: preliminaries}

\subsection{Geometry over $p$-adic rings}

\revstart
Unless mentioned otherwise, we fix a prime $p$ throughout the paper and let $\A$ be the ring of integers of a finite extension of $\bbQ_p$, with maximal ideal $\mathfrak m_\F$ generated by a uniformizer $\pi_\F$, fraction field $\F$ and residue field $\A/\mathfrak m_\F= \RF$ of size $q$, a power of $p$ (and use similar notation for extensions $\F \subset \L$). We will think of $\A$ as fixed and vary over extensions $\A \subset \B$ with residue field extension of degree $m = [\RL:\RF]$. We define $|\cdot |_{\L}$ to be the non-archimedean metric on $\L$ (so that $|\pi_\L|_\L = q^{-m}$) and $\frob{\L}$ to be the Frobenius acting on the unramified closure of $\L$. We denote the Frobenius element in $\Gal(\overline{\mathbb F}_q/\mathbb F_q)$ by $\frob{q}$.
\revend

We fix $\SS= \Spec\B$ and define a variety $X \to \SS$ to be a finite type, geometrically reduced scheme that is flat over $\SS$. By $\dim X$, we mean the relative dimension over $\SS$. Given an extension of schemes $T \to \SS$, we denote the base change by $X_T$ and similarly for morphisms between varieties. For a ring $R$, we let $X(R)$ denote the $\Spec R$ valued points of $X$. 

Given a scheme $X\to\Spec \B$ and a point $x \in X(\B)$, we define $\cO_{X,x}$ to be the localization at the maximal ideal $\mathfrak m_{X,x}$ corresponding to the reduction $\overline{x} \in X(\B/\mathfrak m_\B)$ and $\widehat{\cO}_{X,x}$ to be its completion at its maximal ideal. When $X$ is smooth, a complete set of local co-ordinates $t_1,\dots,t_{\dim X}$ at $x$ determines an isomorphism $\widehat\cO_{X,x} \cong \B[\![t_1,\dots,t_{\dim X}]\!]$ and consequently, 
\[\widehat{\cO}_{X,x}(\B) = \widehat{\cO}_{X,x}(\L) = \mathfrak m_B^{\dim X}.\]
Here, $\widehat{\cO}_{X,x}(\B)$ is taken to be the space of \emph{continuous} ring homomorphisms to $\B$.

\subsection{$p$-adic integration.}\label{section: canonical measure}

We will use techniques of $p$-adic integration throughout this paper. In this section, we give a brief introduction and define our normalizations. We fix an extension $K \subset L$ with residue field $\mathbb F_{q^m}$ and $U \to \Spec \B$ a smooth variety of dimension $d$ in this section. The set $U(\L)$ carries the natural structure of a \emph{$\L$-analytic manifold} \cite[\S2.4]{igusa2000introduction}.

\begin{remark}
    In fact, we need not restrict to $U$ being a smooth variety, the theory works without modification for $U$ an algebraic space. However, for simplicity we stick to the case of $U$ being a scheme here and refer to \cite[\S 2.1]{groechenig2020geometric} for the details in the case of $U$ being an algebraic space.
\end{remark}

\begin{definition}[$p$-adic integration with respect to a form]
Let $\omega \in \Omega_U^d$ be a top form. For $V \subset U(\L)$ a compact subset, one can define the integral
\[\mu_{\omega}(V) \coloneqq \int_{V}|\omega| \in \mathbb R_{\geq 0}.\]
If $\omega = f(x)dx_1\wedge\dots\wedge dx_n$ for a $\L$ analytic function $f$ and local co-ordinates $x_1,\dots,x_n$ on $V$, then
\[\int_{V}|\omega| = \int_{V'}|f(x)|d\mu_{\L^d} \in \mathbb R_{\geq 0}\]
where we identify $V$ with an open subset $V'$ of $\L^d$ using the $x_i$ and $\mu_{\L^d}$ is the canonical Haar measure on $\L^d$ normalized so that $\mu_{\L^d}(\B^d) = q^{md}$. Note that this normalization is slightly non-standard.

\end{definition}

\begin{definition}[Measures from compactifications]\label{defn: measures from scheme compactifications}

Given an open immersion $U \subset X$ over $\Spec\B$ with $X$ a smooth and proper variety, we can assign a canonical measure $\mu_X$ on the set $U(\L)$ \cite[\S 4.1]{yasuda2017wild}. We cover $X$ by Zariski opens $U_i$ such that $U_i(\B)$ covers $X(\B) = X(\L)$ and moreover trivializes $\Omega_X^d|_{U_i}$. Let $\omega_i \in \Omega_X^d(U_i)$ be gauge forms, i.e., nowhere vanishing sections. While these $\omega_i$ might not glue together on intersections, the transition functions are nowhere vanishing on $(U_i\cap U_j)(\B)$ and hence have $p$-adic norm $1$. Therefore, the measures $\mu_{\omega_i}$ on the $p$-adic analytic sets $U_i(\B)$ glue together to give a unique measure $\mu_X$.
\end{definition}

We note that our normalization is such that $\mu_X(U(\L)) = |X(\mathbb F_{q^m})|$ which is \emph{not} the standard one found in the literature. 

The following is easy but very useful.

\begin{thm}(compare \cite[Proposition 2.1]{groechenig2020geometric})\label{thm: change of variables}
Suppose we have a top form $\omega \in \Omega_U^d$ and $V \subset U(\L)$ a compact subset so that $\mu_\omega(V)$ is well defined. Then
\begin{enumerate}
\revstart
    \item If $V \subset Z(\L)$ for a Zariski-closed subscheme $Z \subset U$ of positive codimension in every component of $U$, then $\mu_\omega(V) = 0$.
    \item Suppose that $f: \widetilde U \to U$ is a finite \'etale map of degree $r$. For $y \in U(\L)$, put $N_f(y)=|f^{-1}(y)(\L)|$. Then
    \[\int_{f^{-1}(V)}|f^*\omega| = \int_V N_f(y)|\omega|,\]
    where $f^{-1}(V)$ denotes the preimage in $\widetilde U(\L)$.
\revend
\end{enumerate}
\end{thm}
\begin{proof}
\revstart
For the first part, we work in analytic charts, where the zero set of a nonzero analytic function has Haar measure zero. This follows by induction on the dimension and Fubini's theorem from the corresponding assertion in one variable. For the second part, the inverse function theorem gives a local analytic isomorphism at each $\L$-rational point of a fibre. Since $f$ is finite, it is proper, so after shrinking a neighbourhood of $y$ there are exactly $N_f(y)$ such sheets and no others. The usual change of variables formula on each sheet, followed by a disjoint measurable partition of $V$ into these neighbourhoods, gives the identity.
\revend
\end{proof}
\subsection{Palindromic forms}

Let $K_0(\mathrm{var})$ be the Grothendieck ring of varieties over $\mathbb F_q$. We consider the extension given by 
\[\widetilde{K_0(\mathrm{var})} = \left(K_0(\mathrm{var})\otimes\mathbb Q\right)\left[\frac{\mathbb L^k - \mathbb L}{\mathbb L^{k}-1}\right]_{k\geq 2}\]
where $\mathbb L = [\mathbb A^1]$ is the Lefschetz motive. This ring has implicitly and explicitly appeared before in the literature on motivic integration (\cite{denef1991functional}, \cite[\S 7]{denef_loeser_2000}).

\begin{definition}\label{defn: palindromic forms}

\revstart
The motivating examples are the functions $m \to |X(\mathbb F_{q^m})|$ for $X \in \widetilde{K_0(\mathrm{var})}$. We use the larger class of finite $\overline{\mathbb Q}$-linear combinations of products of functions
\[m \to \frac{q^{mk}-q^m}{q^{mk}-1}\quad (k\geq2), \qquad m \to \alpha^m,\]
where $\alpha$ is a Weil number relative to $q$,\footnote{A Weil number relative to $q$ is a nonzero algebraic integer $\alpha$ such that there exists a $w \in \mathbb Z_{\geq0}$ so that under every complex embedding $\sigma: \overline{\mathbb Q} \to \mathbb C$, $|\sigma(\alpha)| = q^{w/2}$.} and require that their values lie in $\mathbb Q$ for every $m\geq1$. Equivalently, after clearing the local denominators and combining equal exponential bases, the coefficient vector is Galois-stable: $c_{\sigma(\alpha)}=\sigma(c_\alpha)$.
\revend

\revstart
We extend these expressions to $\mathbb Z$ by the same formulas, assigning the local factor at $m=0$ its removable value $(k-1)/k$. This extension is independent of the expression. Indeed, we first cancel $q^m-1$ in each local factor and clear a common denominator $D(q^m)$, where $D(t)$ is a product of polynomials $1+t+\dots+t^{k-1}$. The difference of two expressions then becomes a finite sum $\sum_j c_j\alpha_j^m$ with distinct nonzero bases. If it vanishes for all $m\geq1$, the Vandermonde determinant gives $c_j=0$ for every $j$. Since $D(q^m)\neq0$ for every integer $m$, the extensions agree on $\mathbb Z$. The same argument proves the equivalence with Galois stability and shows that all the extended values are rational.
\revend

\revstart
We say that such a function has weight $k$ if its formal extension satisfies
\revend
\[\rho(-m) = q^{-mk}\rho(m)\]
and we call such functions \emph{palindromic forms of weight $k$}.
\end{definition}

\begin{example}\label{ex: roots of untiy are palindromic forms}
\revstart
Consider the class of the scheme $\cP_k = \Spec \mathbb F_{q^k} \in \widetilde{K_0(\mathrm{var})}$. Its point count is $k$ when $k\mid m$ and zero otherwise, since these points are the $\mathbb F_q$-embeddings of $\mathbb F_{q^k}$ into $\mathbb F_{q^m}$. Writing $\cI$ for the indicator function and $\zeta_k$ for a primitive $k$-th root of unity, we obtain
\[m \to |\cP_k(\mathbb F_{q^m})| = k\cdot\cI(k\mid m) = \sum_{a=1}^k\zeta_{k}^{am}.\]
\revend
This is a palindromic form of weight $0$.
\end{example}

\begin{remark}\label{rmk: palindromic forms closed under sums and products}
It is immediate that the sum of two palindromic forms of weight $k$ has weight $k$ while the product of palindromic forms of weights $k_1,\dots,k_r$ has weight $k_1+\dots+k_r$.
\end{remark}

\begin{lemma}\label{lemma: point counts have right weight}
\revstart
Let $X/\mathbb F_q$ be a smooth, proper variety of pure dimension $d$, not necessarily geometrically connected. The function $m \to \rho_X(m) = |X(\mathbb F_{q^m})|$ is a palindromic form of weight $d$.
\revend
\end{lemma}
\begin{proof}
\revstart
We write $H^i=H^i_{\text{\'et}}(X\times_{\mathbb F_q}\overline{\mathbb F}_q,\mathbb Q_\ell)$ for $\ell$ coprime to $q$. By the Grothendieck-Lefschetz trace formula and the Weil conjectures, the Frobenius eigenvalues $\alpha_{ij}$ are Weil numbers and
\[\rho_X(m) = \sum_{i=0}^{2d}(-1)^i\sum_j\alpha_{ij}^m.\]
Cup product followed by the sum of the trace maps on the geometric connected components gives a perfect pairing $H^i\times H^{2d-i}\to\mathbb Q_\ell(-d)$. Frobenius is a similitude with multiplier $q^d$, so duality pairs its eigenvalues, with multiplicity, by $\alpha_{2d-i,j'}=q^d\alpha_{ij}^{-1}$. This comparison does not require Frobenius to be semisimple. Therefore
\[\rho_X(-m)=\sum_{i=0}^{2d}(-1)^i\sum_j\alpha_{ij}^{-m}
=q^{-md}\sum_{i=0}^{2d}(-1)^i\sum_j\alpha_{2d-i,j'}^m
=q^{-md}\rho_X(m).\]
\revend
\end{proof}

On \'etale cohomology, $\frob{q}$ denotes the geometric Frobenius and an automorphism $h$ acts by $(h^{-1})^*$.

\begin{lemma}\label{lemma: twisted point counts have the right weight}
\revstart
Let $X/\mathbb F_q$ be a smooth proper variety of pure dimension $d$, not necessarily geometrically connected, and let $\ell$ be a prime coprime to $q$. For $g \in \mathrm{Aut}_{\mathbb F_q}(X)$ of finite order $r$, the function
\[m \to \rho_{X,g}(m) \coloneqq \sum_{i=0}^{2d}(-1)^i\tr\left((g+g^{-1})\frob{q}^m|H^i_{\text{\'et}}(X\times_{\mathbb F_q}\overline{\mathbb F}_q,\mathbb Q_\ell)\right)\]
is a palindromic form of weight $d$.
\revend
\end{lemma}
\begin{proof}
\revstart
We write $H^i$ as in the preceding proof. Since $g$ is defined over $\mathbb F_q$, it commutes with $\frob{q}$, and its finite order implies that it is semisimple. After extending scalars to $\overline{\mathbb Q}_\ell$, we decompose each $g$-eigenspace further into generalized Frobenius eigenspaces $H^i_{\lambda,\alpha}$. Here $\lambda^r=1$ and $\alpha$ is a Weil number. On such a space, nilpotent parts do not affect the traces of any integral power of Frobenius, so
\[\rho_{X,g}(m)=\sum_{i,\lambda,\alpha}(-1)^i
    \dim H^i_{\lambda,\alpha}\,(\lambda+\lambda^{-1})\alpha^m.\]
This is an expression of the class in Definition \ref{defn: palindromic forms}. Its values for $m\geq1$ are integers: each of the alternating traces of $g^{\pm1}\frob{q}^m$ is the point count of the corresponding twist of $X$ over $\mathbb F_{q^m}$. We will also use the fact that periodic coefficients can be absorbed into such expressions by the Fourier expansion
\[\cI(m\equiv a\pmod h)=\frac1h\sum_{b=0}^{h-1}\zeta_h^{b(m-a)},\]
since a root of unity times a Weil number is again a Weil number.
\revend

\revstart
The pairing given by cup product and the sum of the component trace maps is invariant under $g$, even when $g$ permutes the components. For $T=g^{\pm1}\frob{q}^m$ it has similitude multiplier $q^{md}$. Duality therefore gives
\[\sum_i(-1)^i\tr(T^{-1}\mid H^i)
    =q^{-md}\sum_i(-1)^i\tr(T\mid H^i).\]
Since $(g^{\pm1}\frob{q}^m)^{-1}=g^{\mp1}\frob{q}^{-m}$, adding the two identities proves $\rho_{X,g}(-m)=q^{-md}\rho_{X,g}(m)$.
\revend
\end{proof}

\begin{remark}\label{rmk: individual densities are not palindromic}
We note that the function
\[m \to \sum_{i=0}^{2\dim X}(-1)^i\tr\left(g\frob{q}^m|H^i_{\text{\'et}}(X\times_{\mathbb F_q}\overline{\mathbb F}_q,\mathbb Q_\ell)\right)\]
is, perhaps surprisingly, not a palindromic form in general. Let $E$ be the elliptic curve defined by the equation $y^2 = x^3 + x$ over $\mathbb Z_p$ with $p \equiv 1 \pmod{4}$ and a fixed  $i = \sqrt{-1} \in \mathbb Z_p$. We have an automorphism $g: E \to E$ defined over $\mathbb Z_p$ given by $g(x,y) = (-x,iy)$. Since $g$ has eigenvalues $\pm i$ on $H^1_{\text{\'et}}(\overline{E},\mathbb Q_\ell)$, there exist $a,b \in \mathbb Z$ such that $a^2 + b^2 = p$ and\footnote{For instance, if $p = 5$ with $i \equiv 2 \pmod{5}$, we have $a = 1, b = 2$.}
\begin{align*}
    \nu(m) \coloneqq \sum_{i=0}^{2\dim X}(-1)^i&\tr\left(g\frob{p}^m|H^i_{\text{\'et}}(X\times_{\mathbb F_p}\overline{\mathbb F}_p,\mathbb Q_\ell)\right) = 1 + (i(a+ib)^m -i(a-ib)^m) + p^m.
\end{align*}
Therefore,
\begin{align*}
\nu(-m) =& 1 + (i(a+ib)^{-m} -i(a-ib)^{-m}) + p^{-m}\\
= &p^{-m}\left(p^m + (i(a-ib)^{m} -i(a+ib)^{m}) + 1 \right) \\\neq& p^{-m}\nu(m)
\end{align*}
which proves that $\nu$ is not a palindromic form.
\end{remark}

\begin{example}\label{defn: rho, eta}
For any extension $\A\subset \B$ where $\B$ has residue field $q^m$, we define
\[\rhol_k(m;q) \coloneqq \int_{\mathfrak m_\L}|t^k|_\L dt = \frac{q^m-1}{q^{m(k+1)}-1}\]
and $\eta_k(m;q) = \rhol_k(m;q) - 1$. The function $m \to \eta_k(m;q)$ is a palindromic form of weight $1$:
\[\eta_k(-m;q) = \frac{q^{-m} - q^{-m(k+1)}}{q^{-m(k+1)}-1} = q^{-m}\frac{q^{m(k+1)}-q^m}{1-q^{m(k+1)}} = q^{-m}\eta_k(m;q).\]
\revstart
Note that $\rhol_k,\eta_k$ depend on $\B$ only through the size of its residue field $q^m$. Indeed, $\eta_k(m;q)$ is a function purely of $q^m$ but we write it this way to emphasize that we will think of $q$ as fixed and $m$ as varying. We allow $k=0$, for which $\rhol_0=1$ and $\eta_0=0$. For $s\geq1$, the identity
\[\rhol_k(m;q^s)=(1+q^m+\dots+q^{m(s-1)})\rhol_{s(k+1)-1}(m;q)\]
shows that the local factors with base $q^s$ also belong to the class in Definition \ref{defn: palindromic forms}.
\revend

\end{example}

\begin{definition}\label{defn: palindromic forms as a function of rings}
We say that a function $\rho: \{\text{extensions } \F \subset \L\} \to \mathbb Q$ is a palindromic form of weight $k$ if
\begin{enumerate}
    \item The function depends on $\L$ only through the degree of residue field extensions $m$.
    \item The function $m \to \rho(\L)$ is a palindromic form of weight $k$.
\end{enumerate}
\end{definition}

\subsection{The poset of admissible pairs}\label{section: admissible pairs}

We discuss here a purely group theoretic notion that will become crucial in the proof of Theorem \ref{thm: Main theorem, density fixing inertia and frobenius coset}. Let $G$ be a finite group.

We say that a tuple $\tau = (H_\tau,g_\tau H_\tau)$ consisting of a subgroup and a left coset is an \emph{admissible pair} if $g_\tau$ normalizes $H_\tau$, i.e., $g_\tau H_\tau = H_\tau g_\tau$. The data of an admissible pair $\tau$ is equivalent to the choice of a pair of subgroups $H_{\tau} \subset H_{1,\tau} \subset G$ such that $H_\tau$ is normal in $H_{1,\tau}$ and $C_\tau \coloneqq H_{1,\tau}/H_{\tau}$ is cyclic along with the choice of a distinguished generator $g_\tau$ for this cyclic group. Here, $H_{1,\tau}$ corresponds to the subgroup generated by $g_{\tau}$ and $H_{\tau}$.

We say that $\tau' \leq \tau$ for two admissible pairs if $H_{\tau'} \subset H_{\tau}$ and $g_{\tau'} \in g_{\tau}H_{\tau}$. This is easily checked to be a partial order. There is one maximal element $(G,G)$ while the minimal elements correspond exactly to $(e,g)$ for $e$ the identity subgroup and $g \in G$ any element.

For $\lambda \in G$, we define the conjugate of $\tau = (H,gH)$ by $\lambda$ by
\[\lambda \tau \lambda^{-1} \coloneqq (\lambda H \lambda^{-1}, \lambda gH \lambda^{-1}).\]
We also define the inverse of an admissible pair as
\[(H,gH)^{-1} \coloneqq (H,g^{-1}H).\]
Both conjugation and inversion are easily seen to be automorphisms of the poset of admissible pairs. 

We will use the language of an incidence algebra on a poset. We briefly recall this notion for the convenience of the reader. Given a poset $\cP$, an interval $[a,b]$ in it is the set of elements $c$ such that $a \leq c \leq b$ (where we tacitly assume that $a\leq b$). Its incidence algebra $\cI_{\cP}$ (over any ring $R$) is the ring of functions from the set of intervals to $R$. Addition in $\cI_{\cP}$ is pointwise while multiplication is defined by convolution:
\[\text{For all } \alpha,\beta \in \cI_{\cP}, \hspace{10 mm} (\alpha\ast\beta)([a,b]) = \sum_{a \leq c \leq b}\alpha([a,c])\beta([c,b]).\]
This multiplication need not be commutative. An element $\alpha \in \cI$ is invertible precisely when $\alpha([a,a]) \in R^\times$ for all $a \in \cP$. We will suppress the ring $R$ from the notation, it will usually be taken to be $\mathbb Q$ or $\mathbb R$.

\begin{definition}\label{defn: alpha, beta}
In particular, we will make use of the following elements in $\cI_{\cP}$ for $\cP$ the poset of admissible pairs:
\[\alpha([\tau',\tau]) \coloneqq |\{\lambda \in H_\tau\backslash G : \lambda \tau' \lambda^{-1} \leq \tau\}|,\]
\[\beta([\tau',\tau]) \coloneqq |\{\tau'' \in \cP : \tau'' = \lambda \tau' \lambda^{-1} \leq \tau \text{ for some } \lambda \in G.\}|.\]
and $\gamma([\tau',\tau]) \coloneqq \alpha([\tau',\tau])/\beta([\tau',\tau])$. If $\tau' \leq \tau$, then $\alpha([\tau',\tau]),\beta([\tau',\tau])$ are both at least $1$ since $\tau'$ is conjugate to itself. Therefore, $\gamma([\tau',\tau]) > 0$ when $\tau' \leq \tau$.

Note that the condition in the definition of $\alpha$ indeed only depends on the coset $H_\tau\lambda$. It is also clear that $\alpha$ and $\beta$ are invariant under conjugating the poset by an element of $G$.
\end{definition}

\subsection{Generically finite, Galois maps}

Let $X,Y$ be two geometrically connected, smooth and proper varieties relative to $\Spec \A$.

\begin{definition}[generically finite]\label{defn: generically fintie}
A map $f: X \to Y$ is said to be generically finite if there is some open $U \subset Y$ such that $f: f^{-1}(U) \to U$ is \'etale and finite. There is in fact a maximal such open which we will denote by $\U{f}$ and the preimage of the complement $\Z{f} = f^{-1}(Y - \U{f})$ will be called the exceptional locus. When $f$ is moreover Galois (Definition \ref{defn: generically finite, galois}), this is codimension $1$ in $X$ by the Jacobian criterion (see \ref{defn: Jacobian}). We furthermore require that $\Z{f}$ has no components supported purely over the special fiber of $\Spec \A$ or equivalently, that there is some closed point on $Y$ over $\Spec \A/\pi \A$ over which the map $f$ is \'etale.

\end{definition}

\begin{definition}[generically finite, Galois]\label{defn: generically finite, galois}
Let $G$ be a constant group scheme over $\Spec \A$. A generically finite map $f: X \to Y$ is said to be Galois with Galois group $G$ if the restriction of the map to $\U{f}$ is an \'etale Galois cover with Galois group $G$ and every such automorphism extends to an automorphism of $f$. That is to say,
\[\mathrm{Aut}_\A(f:X\to Y) = \mathrm{Aut}_\A(f: f^{-1}(\U{f}) \to \U{f}) = G.\] 

\end{definition}

\begin{example}
An interesting family of generically finite, Galois maps are the following. Let $C$ be a curve of genus $g \geq 2$. Let $D$ be a degree $g$ divisor on $C$. Then, the morphism
\begin{align*}
    C^g &\to \Jac(C);\\
    (P_1,\dots,P_g) &\to [P_1] + \dots + [P_g] - D
\end{align*}
is an example of a generically finite map. It would be interesting to compute splitting densities for these maps although we don't pursue it in this paper.
\end{example}

This paper will be mainly concerned with such maps and to that end, we state a few preliminary definitions and properties of such maps here.

\begin{definition}[Decomposition group]\label{defn: decomposition group}
Let $P \in \U{f}(\L)$ be a rational point away from the ramified locus. Then, $f^{-1}(P)$ is an \'etale $\L$ algebra and each geometric point $Q \in f^{-1}(P)$ over $\overline{\L}$ defines a homomorphism $\Gal(\overline{\L}/\L) \to G$. We denote the image (known as the decomposition group of $Q$ over $P$) by $D_Q \subset G$. Correspondingly, the image of the inertia group $\Gal(\overline{L}/L^{\mathrm{ur}})$ is denoted by $I_Q \subset G$. The image of the Frobenius coset corresponds to a generating coset $\frob{Q} \in D_Q/I_Q$. We note that the pair $\tau_Q = (I_Q,\sigma_Q)$ is an \emph{admissible pair} as in \S \ref{section: admissible pairs}.

If $P \in \U{f}(\B)$, then $f^{-1}(P)$ is an unramified extension of $\B$ and a choice of $Q$ thus determines a well defined Frobenius element $\frob{Q} \in G$.
\end{definition}

\begin{remark}\label{rmk: conjugacy of stuff}
Different choices of $Q$ over $P \in U(\L)$ correspond to conjugating $D_Q$ by elements in $G$. More precisely, let $(Q_1,\dots,Q_n)$ be the geometric points of $X$ mapping to $P$. This set has a natural action of $G$ and is a $G$-torsor under this action.

If $Q_i = \lambda(Q_{j})$ for some $\lambda \in G$, then $I_{Q_i} = \lambda I_{Q_j}\lambda^{-1}$ and similarly for the decomposition groups and Frobenius cosets. Therefore, we can define $\tau_P$ as the conjugacy class of $\tau_{Q_i}$ by picking any lift $Q_i$ of $P$ to $X$ and similarly,
\[(I_P,D_P) = \{(\lambda I_{Q_i}\lambda^{-1},\lambda D_{Q_i}\lambda^{-1}) : \lambda \in G\}.\]
\end{remark}

The following definition stratifies $P \in \U{f}(\L)$ based on $\tau_P$, i.e., the data of the (conjugacy class) of the inertia group at $P$ along with a Frobenius coset. 

\begin{definition}\label{defn: rho_f,admissible pair}
For $\tau$ an admissible pair for the group $G$, we define
\[\Utype{\tau}(\L) = \{P \in \U{f}(\L) : \tau \in \tau_P\}\]
with measure $\rhotype{\tau}(\L) = \mu_Y(\Utype{\tau}(\L))$. This only depends on the conjugacy class of $\tau$. The $\Utype{\tau}(\L)$ are in fact $\L$-analytic open sets which proves that the measures $\rhotype{\tau}(\L)$ are well defined as the next lemma shows.
\end{definition}

\begin{lemma}
For any admissible pair $\tau$, the $\Utype{\tau}(\L)$ are $\L$-analytic open sets of $Y(\L)$.
\end{lemma}
\begin{proof}
\revstart
We show that the class of the $G$-torsor $f^{-1}(y)$ in $H^1(\L,G)$ is locally constant on $\U{f}(\L)$. Fix $y$ and put $T=f^{-1}(y)$. The finite \'etale scheme over $(\U{f})_\L$ parametrizing $G$-equivariant isomorphisms from $T$ to the fibres of $f$ has an $\L$-point over $y$, given by the identity of $T$. The inverse function theorem gives a local analytic section through this point. Thus nearby fibres are isomorphic as $G$-torsors, so their admissible pairs have the same conjugacy class. This proves the assertion.
\revend
\end{proof}

We obtain a coarser stratification of $P \in \U{f}(\L)$ by remembering only the (conjugacy class of the) inertia group and decomposition group associated to $P$, i.e., we forget the precise generating coset corresponding to the Frobenius coset.

\begin{definition}\label{defn: rho_f,decomp,inertia}
For $\tau$ an admissible pair for the group $G$, we define
\[\Uid{\tau}(\L) = \{P \in \U{f}(\L) : (H_\tau,H_{1,\tau}) \in (I_P,D_P)\}\]
with measure $\rhoid{\tau}(\L) = \mu_Y(\Uid{\tau}(\L))$. As before, this only depends on the conjugacy class of $\tau$.
\end{definition}

We also have a change of variables formula for generically finite maps. We begin by defining the Jacobian.

\begin{definition}[Jacobian]\label{defn: Jacobian}
Let $f:X \to Y$ be a generically finite map over $\Spec \A$. Let $V \subset X$ and $U \subset Y$ be affine opens such that $f(V) \subset U$ and there exist gauge forms $\omega_V,\omega_U$ on $V,U$ respectively (well defined as usual up to a $p$-adic unit). We define the Jacobian of $f$ on $V$ by
\[\Jac(f) = \frac{f^*(\omega_U)}{\omega_V}.\]
It is independent of the choice of local gauge forms up to a unit in $\A$ and in particular, the norm $|\Jac(f)|$ is a well defined function on $X(\F)$.

\revstart
Alternatively, one could view $\Jac(f)$ as the canonical map between the canonical bundles $f^*\mathscr O(K_Y) \to \mathscr O(K_X)$ or equivalently as a section of $f^*\mathscr O(K_Y)^{\vee}\otimes\mathscr O(K_X)$. Specifying local basis $f^*(\omega_U),\omega_V$ for the two line bundles then determines the function $\Jac(f)$ as defined above. It is clear from this description that the vanishing locus of $\Jac(f)$ is precisely the non-\'etale locus of $f$, which is contained in $Z_f$; for a Galois map it equals $Z_f$.
\revend

\end{definition}

\subsection{Galois twists}

We will use Galois twists of generically finite, Galois maps at multiple points in this paper. For convenience, we recall this definition and state some basic properties. We fix $f: X \to Y$ to be a generically finite, Galois map with Galois group $G$. 

\begin{definition}[Galois twists]
We say that $\widetilde{f}: \widetilde{X} \to Y$ defined over $\Spec \A$ is a twist of $f$ if there exists an unramified extension $\F \subset \F'$ such that the map $\widetilde{f}_{\mathcal O_{\F'}}: \widetilde{X}_{\mathcal O_{\F'}} \to Y_{\mathcal O_{\F'}}$ is isomorphic to $f_{\mathcal O_{\F'}}: X_{\mathcal O_{\F'}} \to Y_{\mathcal O_{\F'}}$. As is well known, such twists are classified by $H^1_{\mathrm{et}}(\mathcal O_{\F}, \mathrm{Aut(f: X_{\mathcal O_{\overline\F}} \to Y_{\mathcal O_{\overline\F}})}).$
\end{definition}

Given an element $g \in G$ of order $r$, we can define the twist of $f$ with respect to $g$ corresponding to the element of $H^1_{\mathrm{et}}(\mathcal O_{\F}, \mathrm{Aut}(f))$ obtained from the inflation of the cocyle $Z^1(\mathcal O_{\F}, \mathrm{Aut}(f))$ induced by $\sigma_\F \to g$. We can describe it geometrically by:

\begin{definition}[$g$-twist]\label{defn: unramified twist by an element}
Let $\mathcal O_{\F'}$ be the unique degree $r$ unramified extension of $\A$ with $\Gal(\mathcal O_{\F'}/\A) = \mathbb Z/r\mathbb Z$ generated by $\frob{\F}$. We define $^gX = X\times_{\Spec \A}\Spec \mathcal O_{\F'}/\langle(g^{-1},\frob{\F})\rangle$ as a geometric quotient, where $\frob{\F}$ acts on the second factor through $\Spec \frob{\F}$. As a scheme, it is defined over $\Spec \A$ and admits a map $^gf: {^gX} \to Y$. It is a standard fact that this is a twist of $f$. One sees from the definition that $^gX$ `\emph{is}' $X$ with the Frobenius acting by ${^g\frob{\F}} \coloneqq g\circ\frob{K}$ instead of $\frob{K}$. The cohomology of the twist with its Frobenius action can be identified with
\[(H^i_{\text{\'et}}({^gX}\times_{\mathbb F_q}\overline{\mathbb F}_q,\mathbb Q_\ell) , {^g\frob{\F}}) \cong (H^i_{\text{\'et}}({X}\times_{\mathbb F_q}\overline{\mathbb F}_q,\mathbb Q_\ell) , {g^{-1}\frob{\F}}). \]

\end{definition}

\begin{remark}
We note that the map $^gf: {^gX} \to Y$ does not have a constant automorphism group $G$, instead it has the \'etale automorphism group scheme $^gG$ which ``\emph{is}" $G$ with the Frobenius acting by $h \to ghg^{-1}$.
\end{remark}

\section{On the palindromicity of various natural densities}\label{section: on the palindromicity}

\revstart
This section contains the main theorems of this paper. Under the resolution hypothesis below, we prove that the paired densities $\rhotype{\tau}+\rhotype{\tau^{-1}}$ and certain related densities are palindromic. Before beginning the proofs, we establish some useful definitions and notations. 
\revend

\begin{definition}[Simple normal crossings divisor]\label{defn: sncd}
Let $X/\Spec \B$ be a variety. We say that a divisor $D \subset X$ is a simple normal crossings divisor (sncd) relative to $\Spec \B$ if $X$ is smooth over $\Spec \B$ and
\begin{enumerate}
    \item $D = \bigcup_{i=1}^{r} D_i$ is a decomposition into geometrically irreducible components such that for any subset $I \subset \{1,\dots,r\}$, $D_I \coloneqq \bigcap_{i \in I}D_i$ is either empty or smooth over $\Spec \B$ of relative codimension $|I|$ in $X$. We use the convention that $D_{\emptyset} = X$.
    \item For every $x \in D_I(\B)$, there exists a regular system of parameters $t_i, i\in I$ over $\B$ such that $D_i$ is cut out by $t_i$ in $\cO_{X,x}$.\footnote{i.e., $t_{j}$ is a non zero divisor in $\cO_{X,x}/(t_{i_1},\dots,t_{i_r})$ for $j \not\in \{i_1,\dots,i_r\} \subset I$ and $\cO_{X,x}/(t_{i_1},\dots,t_{i_r})$ is smooth over $\B$.}
\end{enumerate}
\end{definition}

\begin{definition}[Geometric sncd]\label{defn: geometric sncd}
Let $X/\Spec \A$ be a variety. We say that a divisor $D\subset X$ is a geometric sncd if there exists an unramified extension $\A \subset \mathcal O_{K'}$ and $D_{\mathcal O_{K'}} \subset X_{\mathcal O_{K'}}/\Spec \mathcal O_{K'}$ is a sncd relative to $\Spec \mathcal O_{K'}$. In this case, we write $D = \bigcup_{i=1}^{r}D_i$ with the $D_i$ irreducible over $\Spec \A$ (but possibly not geometrically irreducible). 

For each $i$, we define $\F_i$ to be the smallest unramified extension of $\F$ over which 
\[D_{i,\mathcal O_{K_i}} = \bigcup_{j \in J_i}D_{ij.\mathcal O_{K_i}} \text{ with the } D_{ij} \text{ geometrically irreducible (and hence smooth) }.\]
In this case, the Frobenius automorphism $\frob{\F}$ of $\F_i/\F$ will transitively permute the $D_{ij}$ and we define $\deg(\mathcal O_{K_i}/\A) = k_i$. We will consider the set $J_i = \{{1},\dots,{k_i}\}$ to be equipped with the action of the Frobenius $\frob{\F}$ (corresponding to the action on the divisors $D_{ij}$) and denote by $J_i/\frob{\F}^\ell$ the set of orbits for the action of $\frob{\F}^\ell$. At each closed point $x \in X$, we let $t_{i,x} \in \cO_{X,x}$ be a local function cutting out $D_i$ and similarly, we define $t_{ij,x} \in \cO_{X,x}\otimes_{\A}\mathcal O_{K'}$ cutting out $D_{ij}$ such that $\frob{\F}(t_{ij}) = t_{i\frob{\F}(j)}$.

\end{definition}

We note that $D \subset X$ being a geometric sncd is invariant under taking unramified twists as in Definition \ref{defn: unramified twist by an element} since this condition can be checked over the unramified closure.

For our main theorems, we will need to assume the existence of an equivariant resolution. This is often a fairly mild hypothesis, especially if we allow for the characteristic to vary. We formalize this condition in the following hypothesis.

\begin{hypothesis}
\label{hypothesis:res of sing}
\revstart
The tuple $(g,X\to\Spec\A, Z \subset X)$ where $g \in \Aut(X)$ and $Z \subset X$ is a $g$-equivariant closed subscheme has an equivariant resolution, i.e., there exists a smooth variety $\widetilde{X} \to \Spec \A$ with an automorphism $g: \widetilde{X} \to \widetilde{X}$ and a proper, $g$-equivariant birational map $\pi: \widetilde{X} \to X$ such that $\pi$ is an isomorphism over $X\setminus Z$ and $\pi^{-1}(Z)$ is a geometric sncd. 
\revend

\end{hypothesis}

This hypothesis need not be satisfied even if $g = \mathrm{id}, X = \mathbb P^1$ as can be seen by considering $D = V(x^2 - p^2y^2) \subset \mathbb P^1$ for $p$ odd. The minimal resolution of singularities is given in this example by blowing up $X$ at the closed point $V(t, p)$ and $\pi^{-1}(D)$ has a component purely supported in characteristic $p$ and hence is not a relative normal crossings divisor. Nevertheless, it is generally not a restrictive hypothesis.

\begin{remark}\label{rmk: char 0 resolutions}
If $f: X \to Y$ is generically finite with Galois group $G$ over (an open subset) of the ring of integers of a number field, then the above hypothesis for all the tuples $(g,X,Z_f)$ where $g \in G$ is simultaneously satisfied for all but finitely many $p$ since we can equivariantly resolve singularities over characteristic $0$ (for instance, see \cite{abramovich1996equivariant},\cite{encinas1998good} or \cite{bierstone1997canonical}) and \emph{spread out} to all but finitely many primes.
\end{remark}

Before proving the main theorems of this section, we first do a local computation. For any predicate $\mathscr P$ (such as ``$\gcd(k,m) = \ell$"), we define $\cI(\mathscr P)$ to be the indicator function that is $1$ when the predicate is satisfied and zero otherwise.

Since our divisors are not geometrically irreducible, our formulas depend on the divisibility properties of the size of the residue field of $\L$ and we use the indicator functions to concisely express this.

\begin{lemma}\label{lemma: local integrals}
\revstart
Let $D \subset X$ be a geometric sncd as in Definition \ref{defn: geometric sncd} (with the same notation).  For any $e_1,\dots,e_r \in \mathbb Z_{\geq 0}$ and $M_1 \subset \{1,\dots,k_1\},\dots, M_r \subset \{1,\dots,k_r\}$, define $D_{M_i} = \bigcap_{j \in M_i}D_{ij}$. Let $x\in X(\B)$ for an extension $\B\supset\A$, and suppose that $M_i$ is exactly the set of geometric components of $D_i$ through the reduction $\overline{x}$, i.e., $\overline{x}\in D_{ij}$ if and only if $j\in M_i$. We have the identity
\revend
\begin{align*}
    \rhol_{x,M_1,\dots,M_r}(\L) &\coloneqq \int_{\widehat\cO_{X,x}(\B)}|t_{1,x}|_{\L}^{e_1}\dots |t_{r,x}|_{\L}^{e_r}d\mu_{X_{\SS}}\\
    &= \sum_{\ell_1|k_1,\dots,\ell_r|k_r}\prod_{i=1}^r\cI(\gcd(k_i,m)=\ell_i)\left(\prod_{i=1}^r\rhol_{e_i}^{|M_i/\frob{\F}^{\ell_i}|}(m;q^{k_i/\ell_i})\right),
\end{align*}
\revstart
where $m$ is the degree of the residue field of extension of $L/K$ and $\rhol_k(m;q)$ is as in Example \ref{defn: rho, eta}.
\revend
\end{lemma}

\revstart
We observe that $\frob{\L}$ preserves each $M_i$, since $\overline{x}$ is rational over $\mathbb F_{q^m}$ and $M_i$ contains exactly the components through this point.
\revend

\begin{proof}
\revstart
We put $\ell_i=\gcd(m,k_i)$ and let $\L_i=\L K_i$. Then $\L_i/\L$ is unramified of degree $s_i=k_i/\ell_i$, and its residue field has size $q^{ms_i}$. Let $\L'$ be the compositum of the $\L_i$.
\revend

\revstart
Over $\mathcal O_{L'}$, we choose local equations $t_{ij}$ for the components with $j\in M_i$, compatibly with Frobenius. We can do this by choosing an equation for one component over its unramified field of definition and taking its conjugates. The equations for components not through $\overline{x}$ are units, so
\[
    |t_{i,x}|_\L=\left|\prod_{j\in M_i}t_{ij}\right|_\L
\]
on the residue disc. Since the boundary is a geometric sncd, the differentials of the incident $t_{ij}$ are linearly independent in the relative cotangent space. Their span descends to the residue field of $\B$. We choose a complementary basis there and lift it to functions $s_1,\dots,s_v$ over $\B$, where $v=\dim X-\sum_i|M_i|$. The formal inverse function theorem gives a Frobenius-equivariant analytic coordinate isomorphism
\revend
\begin{equation}\label{eqn: co-ord isomorphism}
\revstart     \widehat{\cO}_{X_{\mathcal O_{L'}},x}(\mathcal O_{L'}) \revend
\revstart     \xrightarrow{\sim} \revend
\revstart     \prod_{i=1}^r\prod_{j\in M_i}\mathfrak m_{\L'} \revend
\revstart     \times(\mathfrak m_{\L'})^v, \revend
\revstart     \qquad P\longmapsto ((t_{ij}(P))_{i,j},(s_a(P))_a). \revend
\end{equation}
\revstart Here Frobenius acts on the first coordinates by \revend
\revstart $(z_{ij})\longmapsto(\frob{\L}(z_{i,\frob{\L}^{-1}j}))$ and on the last coordinates coefficientwise. Taking invariants gives \revend
\[
\revstart     \widehat{\cO}_{X,x}(\B) \revend
\revstart     \cong\prod_{i=1}^r\prod_{\lambda\in M_i/\frob{\L}} \revend
\revstart     \mathfrak m_{\L_i}\times(\mathfrak m_\L)^v. \revend
\]
\revstart Each orbit has length $s_i$, so its coordinate is an element of $\L_i$. The coordinate map has unit Jacobian over $\B$: this can be checked after the unramified extension $\mathcal O_{L'}$, where the coordinates above form a regular system of parameters. Thus it preserves the normalized measures, with each residue disc of mass $1$. It also carries the region of integration bijectively onto $\prod_i\prod_{\lambda\in M_i/\frob{\L}}\mathfrak m_{\L_i}\times(\mathfrak m_\L)^v$: the coordinate isomorphism over $\mathcal O_{L'}$ is $\frob{\L}$-equivariant, and taking $\frob{\L}$-fixed points on both sides preserves the bijection. \revend

\revstart
For an orbit $\lambda$, the product of the conjugate parameters is $N_{\L_i/\L}(t_{i\lambda})$. The normalized absolute values satisfy
\[
    |N_{\L_i/\L}(t)|_\L=|t|_{\L_i}.
\]
The integral is therefore
\[
    \prod_{i=1}^r\prod_{\lambda\in M_i/\frob{\L}}
    \int_{\mathfrak m_{\L_i}}|t|_{\L_i}^{e_i}\,dt
    =\prod_{i=1}^r
    \rhol_{e_i}^{|M_i/\frob{\F}^{\ell_i}|}(m;q^{k_i/\ell_i}).
\]
Indeed, writing $Q_i=q^{mk_i/\ell_i}$, the shell of valuation $a\geq1$ has measure $(Q_i-1)Q_i^{-a}$, so its integral is
$(Q_i-1)\sum_{a\geq1}Q_i^{-a(e_i+1)}=(Q_i-1)/(Q_i^{e_i+1}-1)$.
The orbits of $\frob{\L}$ and $\frob{\F}^{\ell_i}$ on $J_i$ agree. Summing over the possible $\ell_i$ with their indicator functions proves the formula.
\revend
\end{proof}

\begin{thm}\label{thm: symmetry for pullback}
Let $X,Y$ be \emph{smooth proper} varieties over $\Spec \A$ with $f: X \to Y$ a generically finite map with $(\mathrm{id},X,Z_f)$ satisfying Hypothesis \ref{hypothesis:res of sing}. The function
\[\L \to \eta_f(\L) = \int_{X(\B)} f_{\SS}^{*}(d\mu_{Y_{\SS}})\]
\revstart
is a palindromic form of weight $\dim Y$. In fact, if $m$ is the degree of the residue field extension of $L$ and we pick a resolution of singularities $\pi : \widetilde{X} \to X$ as in Hypothesis \ref{hypothesis:res of sing} with $\pi^{-1}(Z_f)$ playing the role of $D$ in Lemma \ref{lemma: local integrals} with the same notation as before, we have the identity
\revend
\[\eta_f(\L) = \sum_{\ell_1|k_1,\dots,\ell_r|k_r}\prod_{i=1}^r\cI(\gcd(m,k_i) = \ell_i)\left(\sum_{\cL }\prod_{i=1}^{r}\eta_{e_i}^{|\Lambda_i|}(m;q^{k_i/\ell_i})\left|D_{\cL}(\mathbb F_{q^m})\right|\right)\]
\revstart
where $e_i\geq0$ are the multiplicities of the Jacobian divisor of $f\circ\pi$ along $D_i$ and $\cL$ ranges over tuples of the form $(\Lambda_1,\dots,\Lambda_r)$ with $\Lambda_i \subset J_i/\frob{\F}^{\ell_i}$ and $D_\cL \coloneqq \bigcap_{i=1}^r\bigcap_{\lambda \in \Lambda_i}\bigcap_{j\in\lambda}D_{ij}$.
\revend
\end{thm}
\begin{proof}

\revstart
By Hypothesis \ref{hypothesis:res of sing}, we choose a resolution $\pi:\widetilde X\to X$ that is an isomorphism over $X\setminus Z_f$, with $D=\pi^{-1}(Z_f)$ a geometric sncd. The change of variables formula on this open set, and the fact that its complement has measure zero, give
\[
    \int_{X(\B)}f_{\SS}^*(d\mu_{Y_{\SS}})
    =\int_{\widetilde X(\B)}(f\circ\pi)_{\SS}^*(d\mu_{Y_{\SS}}).
\]
We henceforth write $X$ for the resolution and $f$ for the composite, retaining $D$ for the inverse image of the original boundary.
\revend

\revstart
The composite is \'etale outside $D$, so its ramified locus is contained in $D$. If $D=\bigcup_iD_i$, we let $e_i\geq0$ be the order of vanishing of its Jacobian along $D_i$. The regular local rings of $X$ are factorial, so on a local chart we have
\[
    \Jac(f)=u\prod_i t_{i,x}^{e_i}
\]
with $u$ a unit: after removing these factors, the divisor of the remaining regular function is zero. Thus $|u|_\L=1$ on the residue disc. Some $e_i$ may be zero, since a point over the original branch locus can be unramified for the composite; these components contribute $\eta_0=0$.
\revend

\revstart
Finally, recall the notation from before that $D_i = \bigcup_{j \in J_i}D_{ij}$ over an unramified extension $\mathcal O_{K_i}/\A$ with the $D_{ij}$ smooth (with $\mathcal O_{K_i}$ the minimal such extension). Given $\cM = (M_1,\dots,M_r)$ with the $M_i \subset J_i$, we define by $D_{\cM}^\circ$ the locally closed stratum of $D_{\cM}=\bigcap_i\bigcap_{j\in M_i}D_{ij}$ consisting of points on exactly these geometric components. We carry over other notation from Lemma \ref{lemma: local integrals} with this boundary $D$.
\revend

In order to compute the integral, we pick a set of points $\cS \subset X(\B)$ in bijection with the points of $X(\mathbb F_{q^m})$ under the reduction map. This gives rise to a decomposition of the $p$-adic analytic set $X(\B)$ into discs around each point $s \in \cS$ isomorphic to $\widehat{\cO}_{X,s}$, giving rise to the equality 
\begin{align*}
    \int_{X(\B)}f_{\SS}^{*}d\mu_{Y_\SS} &= \sum_{\cM}\sum_{s \in \cS,\ \overline{s}\in D_{\cM}^\circ}\int_{\widehat\cO_{X,s}(\B)}|t_1^{e_1}\dots t_r^{e_r}|_{\L}d\mu_{X_\SS}.
\end{align*}
By Lemma \ref{lemma: local integrals}, the above is equal to
\revstart
\[\sum_{\ell_1|k_1,\dots,\ell_r|k_r}\prod_{i=1}^r\cI(\gcd(m,k_i) = \ell_i)\left(\sum_{\cM}\sum_{s \in \cS,\ \overline{s}\in D_{\cM}^\circ}\prod_{i=1}^{r}(\eta_{e_i}(m;q^{k_i/\ell_i})+1)^{|M_i/\frob{\F}^{\ell_i}|}\right),\]
where we recall that $\delta_e(m;q) = \eta_{e}(m;q)+1$. Since $\overline{s}$ lies on exactly the components indexed by the $M_i$, these subsets are $\frob{L}$, and hence, $\frob{\F}^{\ell_i}$ stable. Next, we pass to the orbits by defining $\cL = (\Lambda_1,\dots,\Lambda_r)$ with $ \Lambda_i = M_i/\frob{K}^{\ell_i} \subset J_i/\frob{\F}^{\ell_i}$. Summing over such $\cL$ and noting that $|\{s\in\cS:\overline{s}\in D_\cL^\circ\}| = |D_\cL^\circ(\mathbb F_{q^m})|$, we obtain
\revend
\[\sum_{\ell_1|k_1,\dots,\ell_r|k_r}\prod_{i=1}^r\cI(\gcd(m,k_i) = \ell_i)\left(\sum_{\cL \text{ as above}}\left|D^\circ_{\cL}(\mathbb F_{q^m})\right|\prod_{i=1}^{r}\prod_{\lambda \in \Lambda_i}\left(\eta_{e_i}(m;q^{k_i/\ell_i})+1\right) \right).\]
We now expand the inner product as
\[\prod_{\lambda \in \Lambda_i}\left(\eta_{e_i}(m;q^{k_i/\ell_i})+1\right) = \sum_{\Lambda_i' \subset \Lambda_i}\eta_{e_i}^{|\Lambda_i'|}(m;q^{k_i/\ell_i}).\]
We say that $\cL' = (\Lambda_1',\dots,\Lambda_r') \leq \cL$ when $\Lambda_i' \subset \Lambda_i$ for all $i$. Then, the above expansion lets us rewrite the integral as
\[\sum_{\ell_1|k_1,\dots,\ell_r|k_r}\prod_{i=1}^r\cI(\gcd(m,k_i) = \ell_i)\left(\sum_{\cL \text{ as above}}\sum_{\cL' \leq \cL}\prod_{i=1}^{r}\eta_{e_i}^{|\Lambda_i'|}(m;q^{k_i/\ell_i})|D^\circ_{\cL}(\mathbb F_{q^m})|\right).\]
Switching the order of summation, we have
\[\sum_{\ell_1|k_1,\dots,\ell_r|k_r}\prod_{i=1}^r\cI(\gcd(m,k_i) = \ell_i)\left(\sum_{\cL'}\left(\prod_{i=1}^{r}\left(\eta_{e_i}^{|\Lambda_i'|}(m;q^{k_i/\ell_i})\right) \sum_{\cL' \leq \cL}|D^\circ_{\cL}(\mathbb F_{q^m})|\right)\right).\]
\revstart
The formula follows from
\[
    \sum_{\cL'\leq\cL}|D^\circ_\cL(\mathbb F_{q^m})|
    =|D_{\cL'}(\mathbb F_{q^m})|.
\]
We now check the weights, including the fields of definition. For fixed $\ell_i$, put $\ell=\mathrm{lcm}(\ell_1,\dots,\ell_r)$. Every $D_\cL$ descends to $\mathbb F_{q^\ell}$, since its component set is stable under $\frob{q}^\ell$. We extend its point-count sequence by zero when $\ell\nmid m$. By Lemma \ref{lemma: point counts have right weight} over $\mathbb F_{q^\ell}$, this sequence has the form
\[
    \cI(\ell\mid m)\sum_{i,j}(-1)^i(\beta_{ij}^{1/\ell})^m,
\]
where the $\beta_{ij}$ are the Frobenius eigenvalues over $\mathbb F_{q^\ell}$. The chosen $\ell$-th roots are Weil numbers relative to $q$, and the expression is independent of their choice because of the indicator. It has weight $\dim D_\cL$: on multiples of $\ell$ this is Poincar\'e duality over $\mathbb F_{q^\ell}$, and off those multiples both sides vanish. This extension does not alter our formula, since $\gcd(m,k_i)=\ell_i$ for all $i$ implies $\ell\mid m$.
\revend

\revstart
The factors $\eta_e(m;q^s)$ belong to our class by Example \ref{defn: rho, eta} and have weight $s$. Therefore each nonzero product
\[
    \prod_{i=1}^r\eta_{e_i}^{|\Lambda_i|}(m;q^{k_i/\ell_i})
    |D_\cL(\mathbb F_{q^m})|
\]
with this convention has weight
\[
    \sum_{i=1}^r|\Lambda_i|\frac{k_i}{\ell_i}+\dim D_\cL=\dim X.
\]
Indeed, the sum counts the distinct geometric components cutting out $D_\cL$.
\revend

The indicator functions $\cI(\ell|\gcd(m,k))$ (for $\ell|k$) can be written as a sum of roots of unity $\zeta_k$ of order $k$ as 
\[\cI(\ell|\gcd(m,k)) = \frac{1}{k}\sum_{a=1}^{k}\zeta_{k}^{amk/\ell}.\]
Therefore, they are palindromic forms of weight $0$. Since
\begin{equation}\label{eqn: open stratification formula}
    \cI(\gcd(m,k) = \ell) = \cI(\ell|\gcd(m,k)) - \left(1-\prod_{1 < d, d\ell|k}\left(1-\cI(d\ell|\gcd(m,k))\right)\right),
\end{equation}
$\cI(\gcd(m,k) = \ell)$ is also a palindromic form of weight $0$. All together, this shows that our integral is a palindromic form of weight $\dim Y = \dim X$.
\end{proof}

Despite appearances, the function $\eta_f$ does not depend on the resolution $\pi: \widetilde X \to X$ by the birational change of variables formula. We record the following simplification when all components of the exceptional divisor are geometrically irreducible over $\Spec \A$.

\begin{corollary}\label{cor: pullback formula in sncd case}
\revstart
Let $m$ be the degree of the residue field extension of $L/K$. Let $f: X \to Y$ be as in the previous theorem and suppose that there exists a resolution of singularities $\pi: \widetilde{X} \to X$ as in Hypothesis \ref{hypothesis:res of sing}, with $\pi^{-1}(Z_f) = \bigcup_{i=1}^r D_i$ a sncd over $\A$. Then
\revend
\[\eta_f(\L) = \sum_{J \subset \{1,\dots,r\}}\left|D_{J}(\mathbb F_{q^m})\right|\prod_{j\in J}\eta_{e_j}(m;q)\]
\revstart
for the non-negative Jacobian multiplicities $e_j$.
\revend
\end{corollary}

We prove a modification of the above theorem taking into consideration unramified twists, as in Definition \ref{defn: unramified twist by an element}.

Let $\tau = (H,gH)$ be an admissible pair, i.e., $gH=Hg$. As above, we take $f: X \to Y$ to be a generically finite, Galois map with Galois group $G$ over $\Spec \A$. Define $f_{H}: X \to X/H$ to be the natural quotient map. Since $gH = Hg$, $g$ descends to an automorphism of $X/H$.

We suppose that $(g,X/H,f_H(Z_f))$ satisfies Hypothesis \ref{hypothesis:res of sing} and define $\widetilde{X/H}$ to be a $g$-equivariant resolution
\[\pi_H: \widetilde{X/H} \to X/H.\]
Since all the maps involved are $g$-equivariant and twisting by $g$ is functorial, we obtain maps
\[\pi_\tau: {^g\left(\widetilde{X/H}_{\SS}\right)} \to {^g{(X/H)_{\SS}}}\]
where we base change to $\SS$ \emph{first} and then twist by $g$. We denote ${^g\left(\widetilde{X/H}_{\SS}\right)}$ by $\big({^\tau X_{\SS}}\big)$ and the resulting map by ${^\tau f_{\SS}}: {^\tau X_{\SS}} \to Y$. By construction, ${^\tau f_{\SS}}$ is an \'etale map over the unramified open locus $\U{f}$.

\begin{thm}\label{thm: twist symmetry under pullback}
Let $f: X \to Y$ be a generically finite, Galois map with group $G$ between \emph{smooth, proper} varieties with $(g,X/H,f_H(Z_f))$ and $(g^{-1},X/H,f_H(Z_f))$ satisfying Hypothesis \ref{hypothesis:res of sing} as above and define
\revstart
\[\L \to \eta_{\tau}(\L)\coloneqq \int_{{^{\tau^{-1}} X_{\SS}}(\B)}|\Jac({^{\tau^{-1}} f_{\SS}})|d\mu_{{^{\tau^{-1}} X_{\SS}}}.\]
\revend
Then, $\eta_\tau(\L), \eta_{\tau^{-1}}(\L) \in \mathbb Q$ and $\eta_{\tau}(\L) + \eta_{\tau^{-1}}(\L)$ is a palindromic form of weight $\dim Y$.
\end{thm}

\begin{proof}

\revstart
We put $n=\dim Y$ and use one $g$-equivariant resolution $\pi_H:\widetilde{X/H}\to X/H$, which is also $g^{-1}$-equivariant. Let $\cJ$ be the finite set of geometric components $D_j$ of $E=\pi_H^{-1}(f_H(Z_f))$, and put $E_J=\bigcap_{j\in J}D_j$ for $J\subset\cJ$, with $E_\emptyset=\widetilde{X/H}$. Each nonempty intersection is smooth and proper of pure dimension $n-|J|$. We write $e_j\geq 0$ for the multiplicity of $D_j$ in the Jacobian divisor of $\widetilde{X/H}\to Y$, which is supported on $E$: by Hypothesis \ref{hypothesis:res of sing}, $\pi_H$ is an isomorphism off $f_H(Z_f)$, where the quotient map to $Y$ is \'etale. By this support condition and regularity, the Jacobian is locally a unit times the product of the component parameters with exponents $e_j$.
\revend

\revstart We fix $\L$ with residue field $\mathbb F_{q^m}$ and base change \emph{first} to $\B$. In the cohomological convention of Definition \ref{defn: unramified twist by an element}, the Frobenius for the inverse twist defining $\eta_\tau$ is $\Theta_m=g\frob{q}^m$. We use the induced cohomological action on the component labels throughout; its permutation is inverse to the geometric one, and a permutation and its inverse have the same cycles and stable subsets. The irreducible divisors over $\B$ of this twisted boundary correspond to the $\Theta_m$-orbits $\lambda$ on $\cJ$; write $s_\lambda=|\lambda|$ and $e_\lambda=e_j$ for $j\in\lambda$. The latter is well defined because $g$ and Frobenius preserve the Jacobian divisor. We apply the local computation of Lemma \ref{lemma: local integrals} over $\B$ itself, with extension degree $1$ and base residue size $q^m$. The set $M$ of components through a rational reduction point is $\Theta_m$-stable, since this point is fixed by the twisted Frobenius. On its residue disc the local factor is $\prod_{\lambda\in M/\Theta_m}\delta_{e_\lambda}(1;q^{ms_\lambda})$. Expanding $\delta_e=1+\eta_e$ and passing from open strata to closed intersections, exactly as in Theorem \ref{thm: symmetry for pullback}, gives \revend
\begin{equation}\label{eqn: twisted component formula}
    \eta_\tau(\L)=\sum_{\substack{J\subset\cJ\\\Theta_mJ=J}}
\revstart     w_J(\Theta_m;q^m)\sum_i(-1)^i\tr\left(\Theta_m\mid H^i_{\text{\'et}}(E_J,\mathbb Q_\ell)\right),\qquad \revend
\revstart     w_J(\Theta;x)=\prod_{\lambda\in J/\Theta}\eta_{e_\lambda}(1;x^{s_\lambda}), \revend
\end{equation}
\revstart where here the $E_J$ are taken over $\overline{\mathbb F}_q$ and $\ell$ is coprime to $q$. Indeed, a $\Theta_m$-stable subset is a union of its orbits and defines a closed stratum of the twisted resolution over $\mathbb F_{q^m}$, whose point count is the alternating trace displayed above by the Grothendieck-Lefschetz trace formula. In particular, the expression is rational and depends on $\L$ only through $q^m$, also for ramified extensions. \revend

\revstart
We emphasize that we take traces on an individual summand only when it is preserved. More intrinsically, we use
\[
    V^i=\bigoplus_{J\subset\cJ}H^i_{\text{\'et}}(E_J,\mathbb Q_\ell),\qquad
    \tr(\Theta\mid V^i)=\sum_{\Theta J=J}\tr\left(\Theta\mid H^i_{\text{\'et}}(E_J,\mathbb Q_\ell)\right).
\]
Both $g$ and Frobenius act on this direct sum by transporting the labelled summands, and commute. The trace identity holds because a summand moved to a different summand contributes zero to the diagonal.
\revend

\revstart
We now check palindromicity. The duality argument of Lemma \ref{lemma: twisted point counts have the right weight} applies also to disconnected intersections: cup product followed by the sum of the trace maps on their connected components is a perfect pairing, invariant under every automorphism. For a twisted Frobenius on an intersection of dimension $d=n-|J|$, this pairing has similitude multiplier $q^{md}$. If $J$ is stable under $g^{-1}\frob{q}^m$, it is stable under its inverse $g\frob{q}^{-m}$, and duality gives
\[
    \sum_i(-1)^i\tr\left(g\frob{q}^{-m}\mid H^i_{\text{\'et}}(E_J,\mathbb Q_\ell)\right)
    =q^{-m(n-|J|)}\sum_i(-1)^i\tr\left(g^{-1}\frob{q}^m\mid H^i_{\text{\'et}}(E_J,\mathbb Q_\ell)\right).
\]
The inverse permutations have the same orbits, and $\eta_e(1;q^{-ms})=q^{-ms}\eta_e(1;q^{ms})$. Consequently their local factors contribute $q^{-m|J|}$, so that every term acquires the factor $q^{-mn}$.
\revend

\revstart
For completeness, the dependence of the indexing on $m$ is periodic. Let $h$ be the order of Frobenius on $\cJ$. On a residue class $m\equiv a\pmod h$, we group the subsets fixed by $g\frob{q}^a$ into orbits under $\langle g,\frob{q}\rangle$. The local factor is constant on each such orbit, as a rational function of $q^m$, and the corresponding direct sum of cohomology is stable under both $g$ and Frobenius. Its Frobenius eigenvalues are Weil numbers: their $h$-th powers are the eigenvalues on the individual smooth proper intersections over $\mathbb F_{q^h}$. Decomposing into the eigenspaces of the finite-order operator $g$, the traces are finite algebraic linear combinations of Weil powers; this uses traces on generalized Frobenius eigenspaces and does not require Frobenius to be semisimple. Together with
\[
    \cI(m\equiv a\pmod h)=\frac1h\sum_{b=0}^{h-1}\zeta_h^{b(m-a)},
\]
this gives a rational-valued expression of the class in Definition \ref{defn: palindromic forms}, whose formal extension is well defined by the uniqueness argument there. Under $m\to-m$, the subsets and orbits for $g\frob{q}^{-m}$ are exactly those for $g^{-1}\frob{q}^m$, so the preceding duality calculation proves
\[
    \eta_\tau(-m)=q^{-mn}\eta_{\tau^{-1}}(m),\qquad
    \eta_{\tau^{-1}}(-m)=q^{-mn}\eta_\tau(m).
\]
Here $\eta_\tau(m)$ denotes the expression for residue degree $m$. Adding these equalities proves that $\eta_\tau+\eta_{\tau^{-1}}$ is a palindromic form of weight $n$.
\revend
\end{proof}

\begin{thm}\label{thm: incidence identity}
Let $f:X\to Y$ be a generically finite, Galois map with Galois group $G$. For an admissible pair $\tau = (H,gH)$, suppose that the tuple $(g^{-1},X/H,f_H(Z_f))$ satisfies Hypothesis \ref{hypothesis:res of sing}. Then, we have
\begin{equation}\label{eqn: incidence identity}
     {\eta_{\tau}(\L)} = \sum_{\tau' \leq \tau}\rhotype{\tau'}(\L)\gamma([\tau',\tau]),
\end{equation}
where $\rhotype{\tau}$ is as in Definition \ref{defn: rho_f,admissible pair} and $\gamma([\tau',\tau])$ are as in \S\ref{section: admissible pairs}.
\end{thm}
\begin{proof}
Suppose $\tau = (H,gH)$ and for ease of notation, we let $k = g^{-1}$ with $\nu = \tau^{-1}$. Let $\widetilde{U} = {^k f_{\SS}}^{-1}(U_f), \widetilde{U}_\nu = {^k\left( f_{\SS}^{-1}(U_f)/H\right)}$ with quotient map $f_\nu: \widetilde{U}_\nu \to \U{f}$.

Suppose $P \in \U{f}(\L)$ with geometric pre-images $Q_1,\dots,Q_n$ in $X(\overline{\L}) = {^kX}(\overline{\L})$ (where the two sets are identified using the canonical identification with the only difference being the Frobenius action). For such a $Q_i$, we denote the corresponding inertia group and Frobenius coset in $G$ by $I_i$ and $\frob{i}$, respectively. We similarly define the group schemes ${^kI}_i$ and ${^k\frob{i}}$ in $^k G$. In fact, ${^kI}_i = I_i$ under the canonical identification of ${^kG}(\B^{\mathrm{ur}}) \cong G(\B^{\mathrm{ur}}) = G$ over the unramified closure of $\B$ since such a base change does not affect inertia groups. Finally, we let $\tau_{Q_i} = (I_i,\frob{i})$ be the corresponding admissible pair. 

Let $\overline{Q}_i$ be the image of $Q_i$ in $\widetilde{U}_\nu$. It lands in $\widetilde{U}_\nu(\L)$ precisely when both the inertia group and Frobenius fix it. Therefore
\begin{align*}
    \overline{Q}_i \in \widetilde{U}_\nu(\L) &\iff I_i\overline{Q}_i = \overline{Q}_i \text{ and } {^k\frob{i}}(\overline{Q}_i) = \overline{Q}_i\\ 
    &\iff I_i \subset H \text{ and } k\frob{i} \in H
\end{align*}
where the final equivalence follows from the faithfulness of the action of $G$ on $Q_1\dots,Q_n$. Altogether, we see that the image of $Q_i$ in ${^\nu X_{\SS}}$ is $\B$-rational precisely when $\tau_{Q_i} \leq \tau$, i.e., $I_i \subset H$ and $\frob{i} \in gH$. In other words, we have shown that $f_\nu$ is a covering space map with image 
\[f_\nu(\widetilde{U}_\nu(\L)) = \bigsqcup_{\substack{\tau' \leq \tau\\\text{ up to conjugacy }}}\Utype{\tau'}(\L).\]

\revstart
For every such $P$ with some $i$ such that $\tau_{Q_i} \leq \tau$, the $j$ such that $\tau_{Q_j} \leq \tau$ correspond to $\lambda \in G$ such that $\lambda\tau_{Q_i}\lambda^{-1} \leq \tau$. The number of distinct images of these $Q_j$ in ${^\nu X_{\SS}}$ is precisely equal to $\alpha([\tau_{Q_i},\tau])$ since the ones in the same $H$ orbit get identified (recall the definitions of $\alpha,\beta$ in Definition \ref{defn: alpha, beta}).
\revend

\revstart
Therefore, the number $N_{f_\nu}(P)$ of $\L$-rational preimages of every $P\in\Utype{\tau'}(\L)$ is $\alpha([\tau',\tau])$. If we pick a compact open $V \subset \Utype{\tau'}(\L)$ and a differential form $\omega$ on it, we have
\revend
\[\alpha(\tau',\tau)\int_{V}|\omega| = \int_{f_{\nu}^{-1}(V)}|{f_{\nu}}^*\omega|\] 
\revstart
by the fibre-count formula of Theorem \ref{thm: change of variables}. We pick $\omega$ to be a gauge form with respect to the measure $\mu_Y$ and sum over a disjoint compact open cover $V_i$ of the image $f_\nu(\widetilde U_\nu(\L))$, with each $V_i$ contained in one conjugacy-class stratum, to obtain the required identity
\[\sum_{\tau'\leq \tau}\rhotype{\tau'}(\L)\frac{\alpha(\tau',\tau)}{\beta(\tau',\tau)} = \int_{\widetilde{U}_\nu(\L)}|\Jac(f_{\nu})|d\mu_{{^\nu X_{\SS}}}=\eta_\tau(\L).\]
We divide by $\beta(\tau',\tau)$ in the first term because that is precisely the amount we over-count by when we sum over all types $\tau' \leq \tau$ instead of up to conjugacy. The omitted boundary has measure zero by Theorem \ref{thm: change of variables}, which gives the last equality. 
\revend
\end{proof}

\begin{thm}\label{thm: Main theorem, density fixing inertia and frobenius coset}
\revstart
Let $f:X\to Y$ be a generically finite, Galois map with Galois group $G$ between \emph{smooth, proper} varieties and $\tau$ an admissible pair (with respect to $G$). For every admissible pair $\tau' \leq \tau$ or $\tau' \leq \tau^{-1}$, suppose that the tuples $(g_{\tau'},X/H_{\tau'},f_{H_{\tau'}}(Z_f))$ satisfy Hypothesis \ref{hypothesis:res of sing}. Then, the densities $\rhotype{\tau}(\L) \in \mathbb Q$ and the function
\revend
\[\L \to \rhotype{\tau}(\L) + \rhotype{\tau^{-1}}(\L)\]
\revstart
is a palindromic form of weight $\dim Y$, where $\rhotype{\tau}$ is as in Definition \ref{defn: rho_f,admissible pair}. The sets $\Utype{\tau}(\L)$ and $\Utype{\tau^{-1}}(\L)$ are disjoint if $\tau$ and $\tau^{-1}$ are not conjugate, and equal otherwise. Thus the measure of their union is the displayed sum in the first case and half that sum in the second.
\revend
\end{thm}
\begin{proof}

\revstart
We fix $\L$ and consider the finite lower ideal consisting of the pairs $\nu\leq\tau$ or $\nu\leq\tau^{-1}$. Equation \ref{eqn: incidence identity} gives a triangular system on functions on this poset,
\[
    \eta_\nu(\L)=\sum_{\nu'\leq\nu}\rhotype{\nu'}(\L)\gamma([\nu',\nu]).
\]
In any linear ordering extending the partial order, its coefficient matrix is triangular with diagonal entries $\gamma([\nu,\nu])>0$. It is therefore invertible, and its inverse has entries $\gamma^{-1}([\nu',\nu])$, the inverse in the incidence algebra. Consequently,
\[\rhotype{\tau}(\L) = \sum_{\tau'\leq \tau}\eta_{\tau'}(\L)\gamma^{-1}([\tau',\tau]).\]
\revend
This is a rational number since $\eta_{\tau'}(\L) \in \mathbb Q$ (Theorem \ref{thm: twist symmetry under pullback}) and $\gamma([\tau',\tau]) \in \mathbb Q$ (and hence also its inverse in the incidence algebra). By definition, we have $\alpha([\tau',\tau]) = \alpha([\tau'^{-1},\tau^{-1}])$ and similarly for $\beta$ and hence also for $\gamma$ and $\gamma^{-1}$. Therefore, if we sum the above identity with the corresponding one for $\tau^{-1}$, we obtain
\begin{align*}
    \rhotype{\tau}(\L) + \rhotype{\tau^{-1}}(\L) &= \sum_{\tau'\leq \tau}{\eta_{\tau'}(\L)} \gamma^{-1}([\tau',\tau]) + \sum_{\tau'^{-1}\leq \tau^{-1}}{\eta_{\tau'^{-1}}(\L)}\gamma^{-1}([\tau'^{-1},\tau^{-1}])\\
    &=\sum_{\tau'\leq \tau}\left(\eta_{\tau'}(\L) + \eta_{\tau'^{-1}}(\L)\right){\gamma^{-1}([\tau',\tau])}.
\end{align*}
\revstart
By Theorem \ref{thm: twist symmetry under pullback}, we see that $\eta_{\tau'}(\L) + \eta_{\tau'^{-1}}(\L)$ is a palindromic form of weight $\dim Y$ and since the poset of admissible pairs and $\gamma^{-1}$ depend only on the group $G$ and not on $\L$, the above identity shows that $\rhotype{\tau}(\L) + \rhotype{\tau^{-1}}(\L)$ is a palindromic form of weight $\dim Y$.
\revend
\end{proof}

As an immediate corollary, we have that the measures in Definition \ref{defn: rho_f,decomp,inertia} take rational values and are palindromic forms of the correct weight.

\begin{corollary}\label{cor: symmetry for fixed inertia, decomp}
\revstart
Let $f: X \to Y$ be a generically finite, Galois map between \emph{smooth, proper} varieties. For every admissible pair $\tau' \leq \tau$ or $\tau' \leq \tau^{-1}$, suppose that the tuples $(g_{\tau'},X/H_{\tau'},f_{H_{\tau'}}(Z_f))$ satisfy Hypothesis \ref{hypothesis:res of sing}. 
\revend

Then the measures $\rhoid{\tau}(\L)$ are palindromic forms of weight $\dim Y$.
\end{corollary}
\begin{proof}
\revstart
We put $H=H_\tau$ and $H_1=H_{1,\tau}$. For every generating coset $\theta H$ of $H_1/H$, let $\tau_\theta=(H,\theta H)$, and let $\mathcal C$ be the set of distinct $G$-conjugacy classes of these admissible pairs. Then
\[
    \Uid{\tau}(\L)=\bigsqcup_{[\tau_\theta]\in\mathcal C}\Utype{\tau_\theta}(\L),
    \qquad
    \rhoid{\tau}(\L)=\sum_{[\tau_\theta]\in\mathcal C}\rhotype{\tau_\theta}(\L).
\]
\revend

\revstart
We check that the preceding theorem applies to every $\tau_\theta$. Write $\theta H=g_\tau^aH$ and choose an integer $r$ coprime to $|G|$ with $ar\equiv1\pmod{|H_1/H|}$; such a choice exists by the Chinese remainder theorem. If $(H',uH')\leq\tau_\theta$, then $(H',u^rH')\leq\tau$. Since $\langle u^r\rangle=\langle u\rangle$, a $u^r$-equivariant resolution supplied by the hypothesis is also $u$-equivariant. This applies to every generator, including $\theta^{-1}$.
\revend

\revstart
Inversion permutes $\mathcal C$, so
\[
    2\rhoid{\tau}(\L)=
    \sum_{[\tau_\theta]\in\mathcal C}
    \left(\rhotype{\tau_\theta}(\L)+\rhotype{\tau_\theta^{-1}}(\L)\right).
\]
This identity also holds for classes fixed by inversion. Each summand has weight $\dim Y$ by the preceding theorem, proving the corollary.
\revend
\end{proof}

As the next theorem shows, the hypothesis on the resolution of singularities above are automatically satisfied at almost all primes.

\begin{thm}\label{thm: main theorem for almost all primes}
\revstart
Let $M$ be a number field and $T = \Spec \mathcal O_M$. Let $f:X\to Y$ over $T$ be a generically finite, Galois map with $X_M$ and $Y_M$ geometrically connected, smooth proper varieties over $M$.
\revend

Then, for almost all primes $\mathfrak p$ and every admissible pair $\tau$, the function
\[\L \to \rhotype{\tau}(\L) + \rhotype{\tau^{-1}}(\L)\]
\revstart
is a palindromic form of weight $\dim Y$, as $\L$ ranges over finite extensions of the fraction field of $\widehat{\mathcal O}_{M,\mathfrak p}$.
\revend
\end{thm}
\begin{proof}
\revstart
For each admissible pair $\tau=(H,gH)$, we resolve the pair $(X_M/H,f_H(Z_{f_M}))$ equivariantly under $\langle g\rangle$. The complement of the boundary is finite \'etale over $\U{{f_M}}$ and hence smooth. Canonical resolution followed by principalization of the boundary ideal therefore gives a smooth proper resolution $\pi_H$ that is an isomorphism off the boundary and whose reduced inverse image of the boundary is a geometric sncd. These constructions commute with automorphisms, so the resolution is $g$-equivariant; see \cite{bierstone1997canonical} for canonical resolution and principalization.
\revend

\revstart
There are only finitely many admissible pairs. We spread out these resolutions, their group actions, and their boundary components simultaneously over an open subset of $T$, also shrinking so that $X$ and $Y$ are smooth and proper. After passing to finite fields of definition of the geometric components and excluding their ramified primes, these components and all their intersections are smooth of the required codimensions. We also spread out the isomorphisms off the boundaries. Thus, after deleting finitely many primes, each resolution satisfies Hypothesis \ref{hypothesis:res of sing} over every completed local ring. The composite maps to $Y$ are \'etale off the boundaries, so their Jacobian divisors are supported there. Theorem \ref{thm: Main theorem, density fixing inertia and frobenius coset} now gives the assertion.
\revend
\end{proof}

We also note that as the size of the residue field tends to infinity, the densities $\rhotype{\tau}(L)$ simplify drastically.

\begin{corollary}\label{cor: q to infty limiting behaviour}
We maintain the hypothesis of the above theorem. As $N(\mathfrak p) \to \infty$, we have the following limiting behaviour
    \[  \lim_{N\mathfrak p \to \infty}\frac{\rhotype{\tau}(\widehat{\mathcal O}_{M,\mathfrak p})}{(N\mathfrak p)^{\dim X}} = \begin{cases} 0 & \text{ if } |H_\tau| > 1\\
    \frac{|[g_\tau]|}{|G|} & \text{ otherwise}\end{cases}  .\]
where $G$ is the Galois group of $f$ and $[g_\tau]$ denotes the $G$-conjugacy class of $g_\tau$.
\end{corollary}
\begin{proof}
\revstart
The points $P \in Y(\widehat{\mathcal O}_{M,\mathfrak p})$ where the fiber $f^{-1}(P)$ has a non trivial action of the inertia group are all contained in the locus of points whose reductions modulo $\mathfrak p$ lie in the image of the ramification locus $f(Z_f)(\mathbb F_{\mathfrak p})$. Therefore, their measure divided by $|Y(\mathbb F_{\mathfrak p})|$ is bounded above by
\[\frac{|f(Z_f)(\mathbb F_{\mathfrak p})|}{|Y(\mathbb F_{\mathfrak p})|} \to 0 \text{ as } N\mathfrak p \to \infty\]
by the Lang-Weil estimates. Since also $|Y(\mathbb F_{\mathfrak p})|/(N\mathfrak p)^{\dim X}\to1$, this proves the first part of the corollary.
\revend

\revstart
Similarly, for the second part, we can assume that the map $f: X \to Y$ is \'etale and Galois by excising the ramified locus since it has normalized measure tending to $0$ in the limit $N\mathfrak p \to \infty$. In this case, the problem reduces to the finite field version by Hensel lifting and the second claim follows from the Chebotarev density theorem over finite fields in its uniform form, i.e., from the Lang-Weil estimates applied to the twists ${^gX}$.
\revend
\end{proof}
 
\section{A conjecture on the density of polynomials with fixed factorization type}\label{section: factorization types density}

\revstart
In this section, we prove the tame case of the $\rho$-part of \cite[Conjecture 1.2]{bhargava2022density} for finite extensions of $\mathbb Q_p$, as an application of the ideas in this paper. We recall some notation from their paper first.
\revend

\subsection{Factorization types of polynomials}

A \emph{factorization type of degree $n$} is a multiset $\{f_1^{e_1}f_2^{e_2}\dots f_r^{e_r}\}$ where the $f_i,e_i$ are positive integers satisfying $\sum_i f_ie_i = n$. We allow repeats in the list of symbols $f_i^{e_i}$ but the order in which they appear does not matter. We will often omit exponents $e_i$ if $e_i = 1$. 

For an \'etale extension $L/K$ of degree $n$ over a local field $K$, we define its factorization type to be $(f_1^{e_1},\dots,f_r^{e_r})$ if a uniformizer of $K$ factors in $L$ as $P_1^{e_1}\dots P_r^{e_r}$ where the $P_i$ are primes of $L$ having residue degree $f_i$. 

We consider non-zero polynomials $h(z)$ of degree $n$ in $\A[z]$ as elements in $\A^{n+1}$ through their coordinates with the associated Haar measure $\mu_{\mathrm{Haar}}$ normalized so that the total measure is $1$. Similarly, we consider monic polynomials as elements in $\A^n$ through their non-leading coefficients, and those congruent to $z^n$ modulo $\mathfrak m_K$ as elements in $\mathfrak m_K^n$, in each case with the Haar measure normalized so that the total measure is $1$.

\begin{conjecture}[Conjecture 1.2, \cite{bhargava2022density}, generalized to a local field]\label{conj: bhargava}
Let $\sigma$ be any factorization type of degree $n$ and $K$ a local field with size of residue field $q$. Set
\begin{itemize}
    \item $\rho(n,\sigma;q) \coloneqq$ the density of polynomials $h \in \A[z]$ of degree $n$ such that $L = K[z]/h(z)$ is \'etale over $K$ with factorization type $\sigma$,
    \item $\alpha(n,\sigma;q) \coloneqq$ the density of monic polynomials $h \in \A[z]$ of degree $n$ such that $L = K[z]/h(z)$ is \'etale over $K$ with factorization type $\sigma$,
    \item $\beta(n,\sigma;q) \coloneqq$ the density of monic polynomials $h \in \A[z]$ of degree $n$ with $h(z) \equiv z^n \pmod{\mathfrak m_K}$ such that $L = K[z]/h(z)$ is \'etale over $K$ with factorization type $\sigma$.
\end{itemize}
Then $\rho(n,\sigma;q), \alpha(n,\sigma;q)$ and $\beta(n,\sigma;q)$ are rational functions of $q$ and satisfy
\begin{equation}\label{eqn: symmetry for rho}
    \rho(n,\sigma;q^{-1}) = \rho(n,\sigma;q);
\end{equation}
\begin{equation}\label{eqn: symmetry for alpha}
    \alpha(n,\sigma;q^{-1}) = \beta(n,\sigma;q).
\end{equation}
Implicitly, we have made the claim that the densities depend on $K$ only through the size of its residue field $\mathbb F_q$.
\end{conjecture}

\revstart
We prove Equation \ref{eqn: symmetry for rho} for all factorization types $\sigma$ over finite extensions of $\mathbb Q_p$ at all tame primes (Theorem \ref{thm: symmetry follows from rationality for factorization densities}) in the rest of this section. It should be possible to also prove Equation \ref{eqn: symmetry for alpha} by extending our methods although we do not do so here.
\revend

Let us first relate the above conjecture to the densities appearing in the rest of this paper. To every squarefree polynomial $h(z)$ with coefficients in a local field $K$ associated to a point in $\mathbb P^n(K)$ through its coefficients, we can associate a conjugacy class of an admissible pair $[\tau_h]$ (for $S_n$) using the map 
\[f: (\mathbb P^1)^n \to (\mathbb P^1)^n/S_n = \mathbb P^n.\]
If we further fix an ordering of the roots of $h(z)$, that corresponds to fixing an element in the conjugacy class $[\tau_h]$. For an admissible pair $\tau$, recall that
\[\Utype{\tau}(K) = \{h(z) \in \mathbb P^n(K) \text{ squarefree with } \tau \in [\tau_h]\}\]
with density $\rhotype{\tau}(K) = \mu_{\mathbb P^n}(\Utype{\tau})(K)$. The Haar measure on $\A^{n+1}$ relates to the canonical measure on $\mathbb P^n(\A)$ as follows.

\begin{lemma}\label{lemma: bhargava densities to our densities}
Consider the canonical quotient map
\[\pi: \A^{n+1} \setminus \{0\} \to \mathbb P^n(\A).\]
For a measurable set $S \subset \mathbb P^n(\A)$, we have the equality of measures
\[\mu_{\mathrm{Haar}}(\pi^{-1}(S)) = \frac{\mu_{\mathbb P^n}(S)}{|\mathbb P^n(\mathbb F_q)|}\]
where $q$ is the size of the residue field of $K$.
\end{lemma}
\begin{proof}
In brief, the canonical top form on $\A^{n+1}$ integrated over a $\mathbb G_m$ orbit (acting diagonally) gives a gauge form on $\mathbb P^n(\A)$. Therefore, the measures of $S$ and $\pi^{-1}(S)$ are equal to a global normalization. More explicitly:

Consider local coordinates where $v_p(a_0) \leq v_p(a_i)$ for all $0 < i \leq n$. On this chart, we can realize $\pi$ as a projection 
\[\pi: \A^{n+1} \setminus \{0\} \to \A^n; (a_0,\dots,a_n) \to (a_1/a_0,\dots,a_n/a_0)\]
onto the final $n$ co-ordinates. From this description, one computes that the preimage of a set $S$ has measure exactly
\[\mu_{\mathrm{Haar}}(\pi^{-1}(S)) = \frac{\mu_{\mathbb P^n}(S)}{q^{n}}\left(\frac{q-1}{q} + \frac{q-1}{q^{n+2}} + \dots\right) = \frac{\mu_{\mathbb P^n}(S)(q-1)}{q^{n+1}-1}\]
as required.
\end{proof}

In the other direction, to each admissible pair $\tau$ as above (with respect to $S_n$), we can associate a factorization type $s(\tau)$ such that if $\tau \in [\tau_h]$ for some squarefree polynomial $h(z)$, then $s(\tau)$ is the factorization type for $h(z)$:

\begin{definition}[Factorization types from admissible pairs]\label{defn: factorization types from admissible pairs}
Recall that an admissible pair $\tau$ has associated to it a normal inclusion of subgroups of $S_n$: $H_\tau \subset H_{1,\tau}$ where the smaller group corresponds to the inertia group and the bigger group to the total Galois group. The action of $H_{1,\tau}$ on the set $[n] = \{1,\dots,n\}$ will partition $[n]$ into orbits $\pi_1,\dots,\pi_r$ and we define $H_{1,\tau}^{(i)}$ to be the projection of $H_{1,\tau}$ in the symmetric group $S_{\pi_i} = \operatorname{Sym}(\pi_i)$ and similarly for $H_\tau^{(i)}$. Furthermore, the action of $H_{\tau}^{(i)}$ will partition $\pi_i$ into further orbits $\pi_{i,j}$. Since $H_\tau$ is normal in $H_{1,\tau}$, the $\pi_{i,j}$ all have the same size that we denote by $e_i$ and the number of such orbits will be denoted by $f_i$. Given this data, the factorization type $s(\tau)$ is defined to be the multiset with $r$ elements $f_i^{e_i}$ corresponding to the parts  $\pi_1,\dots,\pi_r$.
    
\end{definition}

Note that $s(\tau)$ only depends on the conjugacy class of $\tau$. We also note in passing that the coset $g_\tau$ cyclically permutes the partitions $\pi_{i,j}$ for $j=1,\dots,f_i$. The next lemma relates the factorization type densities to the splitting type densities from earlier in the paper and along the way, shows that $s([\tau_h]) = \sigma_h$ as claimed above.

\begin{lemma}\label{lemma: relation of bhargava conjecture to our theorem}
Let $\sigma = \{f_1^{e_1},\dots,f_r^{e_r}\}$ be a factorization type. Then, we have the identity
\[\rho(n,\sigma;q) = \frac{1}{|\mathbb P^n(\mathbb F_q)|}\sum_{\substack{\tau \text{ up to } S_n \text{ conjugacy}\\ \text{satisfying } s(\tau) = \sigma }}\rhotype{\tau}(K).\]
\end{lemma}
\begin{proof}\label{pf: bhargava conjecture}

We will show that the polynomials contributing to the density calculation in the right hand side correspond precisely to the polynomials contributing to the density calculation on the left hand side. This will follow if we show that a polynomial $h(z)$ with admissible pair $[\tau_h]$ satisfies $s([\tau_h]) = \sigma_h$. Since we are taking admissible pairs up to conjugacy on the right, this will prove exactly the bijective correspondence we need. The normalizing factor $|\mathbb P^n(\mathbb F_q)|^{-1}$ follows from Lemma \ref{lemma: bhargava densities to our densities}.

To this end, let $h(z) \in K[z]$ be a squarefree degree $n$ polynomial with factorization type $\sigma$. Upon fixing an ordering of its roots, we obtain an admissible pair $\tau_h$. The corresponding decomposition of $[n]$ into parts $\pi_1,\dots,\pi_r$ as in Definition \ref{defn: factorization types from admissible pairs} above corresponds exactly to the orbits of $\Gal(\overline{K}/K)$ on the set of roots and hence to the decomposition of $h(z) = h_1(z)\dots h_r(z)$ into irreducible polynomials over $K$. Moreover, $H_{1,\tau}^{(i)}, H_{\tau}^{(i)}$ correspond respectively to the Galois group, inertia group of the splitting field $L_i/K$ of $h_i(z)$ and it follows that $f_i,e_i$  are respectively the residue degree and ramification index of the extension $K \subset K[z]/(h_i(z))$.

\end{proof}

We next show that, at tame primes, the factorization densities are given by a single rational function of the size of the residue field. The rationality results of \cite{del_corso_dvornicich_2000} and \cite{john} are proved for a fixed absolute ramification index and residue degree of $K$; the uniformity in these two invariants comes from the general theorems of this paper together with the Chebotarev density theorem over $\mathbb Q$.

\begin{lemma}\label{lemma: uniformity of factorization densities}
\revstart
Let $\sigma$ be a factorization type of degree $n$ and let $\mathcal L_\sigma$ be the set of finite extensions of $\mathbb Q_p$, as $p$ varies, for which the characteristic of the residue field is coprime to the exponents $e_i$ occurring in $\sigma$. There is a rational function $r_\sigma(t) \in \mathbb Q(t)$ such that
\revend
\[\rho(n,\sigma;q) = r_\sigma(q)\]
for every $K \in \mathcal L_\sigma$ with residue field of size $q$.
\end{lemma}
\begin{proof}
By \cite{del_corso_dvornicich_2000} and \cite[Theorem 3.1]{john}, for $K \in \mathcal L_\sigma$ the density $\rho(n,\sigma;q)$ depends on $K$ only through its absolute ramification index $e$ and its residue degree $f$, and for fixed $e$ and $f$ it is given by a rational function $R_{e,f}(t) \in \mathbb Q(t)$ evaluated at $t = q = p^f$, as $p$ ranges over the primes coprime to the $e_i$. We set $r_\sigma = R_{1,1}$, so that $\rho(n,\sigma;p) = r_\sigma(p)$ for every such prime $p$, and we show that $R_{e,f} = r_\sigma$ for all $e$ and $f$. Since a rational function is determined by its values at infinitely many points, it suffices to prove that
\begin{equation}\label{eqn: uniformity at large primes}
    \rho(n,\sigma;q) = r_\sigma(q) \quad \text{for every finite extension } K \text{ of } \mathbb Q_p \text{ with residue field of size } q,
\end{equation}
for all but finitely many primes $p$: given $e$ and $f$, the two rational functions $R_{e,f}$ and $r_\sigma$ then agree at $t = p^f$ for all large $p$, hence are equal.

\revstart
We first prove Equation \ref{eqn: uniformity at large primes} for the unramified extensions of $\mathbb Q_p$. Consider the quotient map $f: X = (\mathbb P^1)^n \to Y = \mathbb P^n$ over $\Spec \mathbb Z$. As in the proof of Theorem \ref{thm: main theorem for almost all primes}, we pick characteristic zero equivariant embedded resolutions, chosen to be isomorphisms outside $f_H(Z_f)$ by the canonical resolution and principalization of \cite{bierstone1997canonical}, and spread them out: there is an integer $N$ such that for every admissible pair $\tau = (H,gH)$ of $S_n$, the tuples $(g^{\pm1},X/H,f_H(Z_f))$ satisfy Hypothesis \ref{hypothesis:res of sing} over $\Spec \mathbb Z[1/N]$. We enlarge $N$ so that the Jacobian divisors are supported on these boundaries, excluding any vertical components. We fix an auxiliary prime $\ell$ and enlarge $N$ so that $\ell\mid N$. For each resolution, let $\cJ$ be the set of geometric components of its boundary $E$, with $E_J=\bigcap_{j\in J}D_j$ and $E_\emptyset=\widetilde{X/H}$ as above. After enlarging $N$ again, the components are parametrized by a finite \'etale scheme over $\Spec\mathbb Z[1/N]$, and every $E_J$ spreads out to a smooth proper scheme over the ring of integers of a field of definition with $N$ inverted, a finite \'etale cover of $\Spec\mathbb Z[1/N]$.
\revend

\revstart
Let $\Gamma$ be the Galois group of the maximal extension of $\mathbb Q$ unramified outside $N$, and for a prime $p \nmid N$ let $\frob{p} \in \Gamma$ denote a geometric Frobenius element at $p$, well defined up to conjugacy. The actions of $\Gamma$ and $g$ on $\cJ$ commute and preserve the multiplicities $e_j$. For each resolution we use the finite-dimensional continuous representation
\[
    V^i=\bigoplus_{J\subset\cJ}H^i_{\text{\'et}}(E_J,\mathbb Q_\ell),
\]
where now $E_J$ is over $\overline{\mathbb Q}$. Galois acts by transporting the labelled summands; equivalently, each orbit of summands is the representation induced from the stabilizer of one $J$. The automorphism $g$ transports the same summands and commutes with this action, since it is defined over $\mathbb Z[1/N]$. Proper smooth base change over the finite \'etale fields of definition shows that these representations are unramified outside $N$, and identifies them with their specializations compatibly with $g$ and Frobenius. Thus this direct sum is a genuine $\Gamma$-representation even when an individual $E_J$ does not descend to $\mathbb Q$.
\revend

\revstart
Let $\L_m$ be the unramified extension of $\mathbb Q_p$ of degree $m$. We apply \ref{eqn: twisted component formula} over $\A=\mathcal O_{\L_m}$, with extension degree $1$, base residue size $x=p^m$ and twisted Frobenius $g\frob{p}^m$. The local factors are $\eta_e(1;x^s)$, where $s$ is a cycle length of this twisted Frobenius. The multiplicities are fixed, and there are finitely many possible cycle types, with $s\leq|\cJ|$; the cycle lengths themselves can vary with $p$ and $m$.
\revend

\revstart
Let $b$ be the maximum of $|\cJ|$ over our finitely many resolutions and let $(e,s)$ range over the finitely many possible multiplicity and cycle-length pairs. We take
\[
    \Delta_0(t)=\prod_{(e,s)}(t^{s(e+1)}-1)^b,\qquad
    \Delta(t)=(1+t+\dots+t^n)\Delta_0(t).
\]
The repetitions allow for several cycles with the same $(e,s)$ in a product. Since $\eta_e(1;x^s)=(x^s-1)/(x^{s(e+1)}-1)-1$, $\Delta_0(x)w_J(\Theta;x)$ is a polynomial in $x$ for every possible orbit configuration. Let $\chi$ be the inverse $\ell$-adic cyclotomic character, so that $\chi(\frob{p})=p$. We define
\[
    \Phi_\tau(\gamma)=
    \left.\sum_{J\subset\cJ}\cI(g\gamma J=J)\,
    \Delta_0(x)w_J(g\gamma;x)\sum_i(-1)^i
    \tr\left(g\gamma\mid H^i_{\text{\'et}}(E_J,\mathbb Q_\ell)\right)\right|_{x=\chi(\gamma)}.
\]
We interpret a summand as zero unless $g\gamma J=J$, when $g\gamma$ is indeed an endomorphism of that summand. We cancel the denominators in $\Delta_0(x)w_J(g\gamma;x)$ before evaluating at $x=\chi(\gamma)$; this defines $\Phi_\tau$ even when $\Delta_0(\chi(\gamma))=0$. The combinatorial data factor through the finite permutation group on $\cJ$. On each such open and closed subset of $\Gamma$, the twisted trace in the $J$-summand is $\tr(P_Jg\gamma\mid V^i)$, where $P_J$ is projection onto that summand, and is continuous. Conjugation in $\Gamma$ transports the subsets $J$, their cycle data and their traces, since $g$ commutes with $\Gamma$. Therefore the total sum $\Phi_\tau$ is a continuous class function. By proper smooth base change and \ref{eqn: twisted component formula},
\[
    \Phi_\tau(\frob{p}^m)=\Delta_0(p^m)\eta_\tau(\L_m).
\]
\revend

\revstart
We apply the M\"obius inversion of \ref{eqn: incidence identity}: its coefficient matrix is triangular with nonzero diagonal, so its inverse has entries $\gamma^{-1}([\tau',\tau])$, rational numbers independent of $\L$, and
$\rhotype{\tau}(\L)=\sum_{\tau'\leq\tau}\eta_{\tau'}(\L)\gamma^{-1}([\tau',\tau])$.
We therefore obtain the same presentation for $\Delta_0(p^m)\rhotype{\tau}(\L_m)$. Summing over the admissible pairs $\tau$ up to conjugacy with $s(\tau)=\sigma$ and using Lemma \ref{lemma: relation of bhargava conjecture to our theorem}, we obtain a continuous class function $\Phi$ on $\Gamma$ with
\revend
\begin{equation}\label{eqn: density as class function}
    \Delta(p^m)\,\rho(n,\sigma;p^m) = \Phi(\frob{p}^m)
\end{equation}
\revstart
for every prime $p\nmid N$ and every $m\geq1$. Here $\rho(n,\sigma;p^m)$ denotes the density over $\L_m$; the factor $1+t+\dots+t^n$ in $\Delta$ accounts for the normalization, so no division by $1+\chi(\gamma)+\dots+\chi(\gamma)^n$ on $\Gamma$ is needed.
\revend

Now write $r_\sigma = P/Q$ with coprime polynomials $P, Q \in \mathbb Q[t]$ and let $S$ be the finite set of primes dividing $N$ or the $e_i$, together with the finitely many primes $p$ for which $Q(p^m) = 0$ for some $m \geq 1$; note that $\Delta(p^m) \neq 0$ for every prime $p$. The functions
\[\gamma \mapsto Q(\chi(\gamma))\Phi(\gamma) \quad \text{and} \quad \gamma \mapsto \Delta(\chi(\gamma))P(\chi(\gamma))\]
\revstart
are continuous class functions on $\Gamma$, and by \ref{eqn: density as class function} with $m=1$ and the case $e=f=1$ of the rationality result, they agree at $\frob{p}$ for every prime $p \notin S$. By the Chebotarev density theorem, the union of the Frobenius conjugacy classes of the primes outside $S$ is dense in $\Gamma$, so the two class functions agree everywhere, and in particular at $\frob{p}^m$ for every $p \notin S$ and $m \geq 1$. Using \ref{eqn: density as class function} once more, this gives $Q(p^m)\Delta(p^m)\rho(n,\sigma;p^m) = \Delta(p^m)P(p^m)$, i.e. $\rho(n,\sigma;p^m) = r_\sigma(p^m)$, which is \ref{eqn: uniformity at large primes} for the unramified extensions of $\mathbb Q_p$ with $p \notin S$.
\revend

\revstart
Finally, let $K$ be an arbitrary finite extension of $\mathbb Q_p$ with $p \notin S$ and residue field of size $q$, and let $K'$ be the unramified extension of $\mathbb Q_p$ with the same residue field. By Theorem \ref{thm: Main theorem, density fixing inertia and frobenius coset}, applied over $\A = \mathbb Z_p$ where Hypothesis \ref{hypothesis:res of sing} holds for every $\tau$ since $p \nmid N$, the sums $\rhotype{\tau}(K) + \rhotype{\tau^{-1}}(K)$ depend on $K$ only through $q$. This dependence also follows directly from \ref{eqn: twisted component formula} and the inversion above, since the geometric sncd and the local integral computation persist under ramified base change. The set of admissible pairs $\tau$ with $s(\tau) = \sigma$ is stable under $\tau \mapsto \tau^{-1}$ because $s(\tau^{-1}) = s(\tau)$, and
\[
    \sum_{s(\tau)=\sigma}\rhotype{\tau}(K)
    =\frac12\sum_{s(\tau)=\sigma}\left(\rhotype{\tau}(K)+\rhotype{\tau^{-1}}(K)\right),
\]
where both sums are over conjugacy classes; this also counts self-inverse classes correctly. Dividing by $|\mathbb P^n(\mathbb F_q)|$ as in Lemma \ref{lemma: relation of bhargava conjecture to our theorem}, we obtain $\rho(n,\sigma;K) = \rho(n,\sigma;K') = r_\sigma(q)$. This proves \ref{eqn: uniformity at large primes} and the lemma.
\revend
\end{proof}

\begin{thm}\label{thm: symmetry follows from rationality for factorization densities}
\revstart
    For a factorization type $\sigma$, let $\mathcal L_\sigma$ be the set of finite extensions of $\mathbb Q_p$, as $p$ varies, for which the characteristic of the residue field is coprime to the exponents $e_i$ occurring in $\sigma$. 
\revend
    
\revstart
    For these finite extensions of $\mathbb Q_p$, $\rho(n,\sigma;q)$ is equal to a rational function $r_\sigma(t)$ evaluated at $t = q$, the size of the residue field. This rational function takes values in $\mathbb Q$ and satisfies the symmetry
\revend
    \[r_\sigma(t^{-1}) = r_\sigma(t).\]
\end{thm}
\begin{proof}
    By Lemma \ref{lemma: uniformity of factorization densities}, there is a rational function $r_\sigma(t) \in \mathbb Q(t)$ with $\rho(n,\sigma;q) = r_\sigma(q)$ for every $K \in \mathcal L_\sigma$.
    It thus suffices to prove the above symmetry after evaluating $t$ at infinitely many values $q$ coming from $\mathcal L_\sigma$  since rational functions are determined by their values at any large enough set of points. 
    
\revstart
    In particular, we fix a sufficiently large prime $p$ coprime to the exponents in $\sigma$ and consider the unramified extensions $\L_m/\mathbb Q_p$ of degree $m$. By Lemma \ref{lemma: relation of bhargava conjecture to our theorem},
    \[
        |\mathbb P^n(\mathbb F_{p^m})|\rho(n,\sigma;p^m)
        =\frac12\sum_{s(\tau)=\sigma}\left(\rhotype{\tau}(\L_m)+\rhotype{\tau^{-1}}(\L_m)\right),
    \]
    where we sum over conjugacy classes, using $s(\tau^{-1})=s(\tau)$, including self-inverse classes. Each summand satisfies the weight $n$ functional equation of Theorem \ref{thm: main theorem for almost all primes}: we use the resolutions constructed in the uniformity proof, which are isomorphisms off their boundaries, and apply Theorem \ref{thm: twist symmetry under pullback} and the inversion of Equation \ref{eqn: incidence identity}. Moreover,
    \[
        \frac{1}{|\mathbb P^n(\mathbb F_{p^m})|}=\delta_n(m;p)
        =1-\frac{p^{m(n+1)}-p^m}{p^{m(n+1)}-1}
    \]
    is a palindromic form of weight $-n$: it belongs to the class in Definition \ref{defn: palindromic forms} and satisfies $\delta_n(-m;p)=p^{mn}\delta_n(m;p)$. Thus the exponential expression for the ratio $m\to r_\sigma(p^m)$ satisfies the weight $0$ functional equation.
\revend
    
\revstart
    To identify its values at negative $m$, we use the uniqueness argument in Definition \ref{defn: palindromic forms}. We write $r_\sigma=P/Q$ and clear $Q(p^m)$ and the denominators in this expression: the resulting difference is a finite sum $\sum_i c_i\alpha_i^m$, with algebraic $c_i$ and nonzero $\alpha_i$, vanishing for every $m\geq1$, and hence for every integer $m$ by the Vandermonde determinant after combining equal $\alpha_i$. Thus the formal extension at $-m$ equals $r_\sigma(p^{-m})$ wherever defined, giving
    \[r_\sigma(p^{-m})=r_\sigma(p^m)\]
    for all but finitely many $m\geq1$; these infinitely many values prove the required identity of rational functions.
\revend
\end{proof}

\revstart
The above theorem proves the tame case of the $\rho$-part of the conjecture for finite extensions of $\mathbb Q_p$, by combining earlier rationality results with the general theorems in this paper. Unfortunately, this is the limit of this direct application of our general theorems in this context since the splitting densities $\rhotype{\tau}(K)$ need not be rational functions of the size of the residue field. In order to complete the proofs along the current lines, we would need to construct a resolution of the ``resultant locus'' \ref{rmk: wild case} uniformly over $\Spec \mathbb Z$. The wild primes are instead treated in the companion paper \cite{companion} by a completely different method, which also produces a new proof of the symmetry -- this time as a consequence of a relationship between splitting densities of monic polynomials, and splitting densities of polynomials whose roots all reduce to 0 $\pmod p$.
\revend

\bibliography{reference}

\end{document}